\documentclass[a4paper]{amsart}
\usepackage{latexsym,bm,stmaryrd}
\usepackage{amsmath,amsthm,amsfonts,amssymb,mathrsfs,pb-diagram}

\usepackage{microtype}
\usepackage{mathtools}
\usepackage{mathbbol,wasysym}

\let\<=\langle
\let\>=\rangle
\def\({\big(}
\def\){\big)}

\def\Z{\mathbb{Z}}
\def\Q{\mathbb{Q}}

\def\N{\mathbb{N}}

\def\O{\mathcal{O}}

\def\Lam{\Lambda}
\def\Sym{\mathfrak{S}}

\newcommand\RR{\mathscr{R}}

\newcommand\HH{\mathscr{H}}

\def\bn[#1,#2]{\begin{bmatrix}#1\\#2\end{bmatrix}}

\def\wnu{\widetilde{\nu}}
\def\fb{\mathbf{b}}
\def\fg{\mathfrak{g}}

\newcommand\bi{\mathbf{i}}
\newcommand\bj{\mathbf{j}}

\DeclareMathOperator\Hom{Hom}

\DeclareMathOperator\df{def}

\DeclareMathOperator\Proj{Proj}

\DeclareMathOperator\Mod{Mod}

\title[]
{Graded dimensions and monomial bases for the cyclotomic quiver Hecke algebras}
\subjclass[2010]{20C08, 16G99, 06B15}
\keywords{Cyclotomic quiver Hecke algebras, categorification}
\author{Jun Hu}\address{Key Laboratory of Algebraic Lie Theory and Analysis of Ministry of Education\\
School of Mathematics and Statistics\\
  Beijing Institute of Technology\\
  Beijing, 100081, P.R. China}
\email{junhu404@bit.edu.cn}
\author{Lei Shi}\address{School of Mathematics and Statistics\\
  Beijing Institute of Technology\\
  Beijing, 100081, P.R. China}
\email{3120195738@bit.edu.cn}

\numberwithin{equation}{section}
\newtheorem{prop}[equation]{Proposition}
\newtheorem{thm}[equation]{Theorem}

\newtheorem{cor}[equation]{Corollary}

\newtheorem{lem}[equation]{Lemma}
\newtheorem{examp}[equation]{Example}

\theoremstyle{definition}
\newtheorem{dfn}[equation]{Definition}
\theoremstyle{remark}
\newtheorem{rem}[equation]{Remark}

\begin{document}

\begin{abstract} In this paper we give a closed formula for the graded dimension of the cyclotomic quiver Hecke algebra $\RR^\Lam(\beta)$ associated to an {\it arbitrary} symmetrizable Cartan matrix $A=(a_{ij})_{i,j}\in I$, where $\Lam\in P^+$ and $\beta\in Q_n^+$. As applications, we obtain some {\it necessary and sufficient conditions} for the KLR idempotent $e(\nu)$ (for any $\nu\in I^\beta$) to be nonzero in the cyclotomic quiver Hecke algebra
$\RR^\Lam(\beta)$. We prove several level reduction results which decompose $\dim\RR^\Lam(\beta)$ into a sum of some products of $\dim\RR^{\Lam^i}(\beta_i)$ with $\Lam=\sum_i\Lam^i$ and $\beta=\sum_{i}\beta_i$, where $\Lam^i\in P^+, \beta^i\in Q^+$ for each $i$. Finally, we construct some explicit monomial bases for the subspaces $e(\wnu)\RR^\Lam(\beta)e(\mu)$ and $e(\mu)\RR^\Lam(\beta)e(\wnu)$ of $\RR^\Lam(\beta)$, where $\mu\in I^\beta$ is {\it arbitrary} and $\wnu\in I^\beta$ is a certain specific $n$-tuple defined in (5.1).
\end{abstract}


\maketitle
\setcounter{tocdepth}{1}
\tableofcontents

\section{Introduction}

The idea of ``categorification'' originates from the work \cite{Cr} and \cite{CF} in their study of quantum gravity and four-dimensional topological quantum field theory. Many important knot invariants (e.g., Jones polynomials \cite{Kh00}) can be categorified and categorification has now become an intensively studied subject in several mathematical and physical areas. For each symmetrizable Cartan matrix $A=(a_{ij})_{i,j\in I}$, Khovanov-Lauda \cite{KL1,KL2} and Rouquier \cite{Rou1, Rou2} introduced a remarkable family of $\Z$-graded algebras $\RR=\bigoplus_{\beta\in Q_n^+}\RR(\beta)$, called quiver Hecke (or KLR) algebras, and used them to categorify the negative parts $U_q(\mathfrak{g})^{-}$ of the quantum group $U_q(\mathfrak{g})$ associated to $A$. For each dominant integral weight $\Lam\in P^+$, they also defined their graded quotients, $\RR^\Lam=\bigoplus_{\beta\in Q_n^+}\RR^\Lam(\beta)$, called cyclotomic quiver Hecke (or cyclotomic KLR) algebras, and conjectured that they can be used to categorify the integrable highest weight module $V(\Lam)$ over the quantum group $U_q(\mathfrak{g})$. The conjecture was proved by Kang and Kashiwara in \cite{KK}. When the ground field $K$ has characteristic $0$ and $A$ is symmetric, Rouquier \cite{Rou2} and Varagnolo-Vasserot \cite{VV} have proved that the categorification sends the indecomposable projective modules over the quiver Hecke algebra $\RR$ to the canonical bases of $U_q(\mathfrak{g})^{-}$.

In many aspects the structure and representation theory of the quiver Hecke algebra $\RR(\beta)$ resemble that of the affine Hecke algebra (\cite{G},\cite{Klesh:book}). For example, Rouquier \cite{Rou2} presented an isomorphism between some localized forms of the quiver Hecke algebra of type $A$ and of the affine Hecke algebra of type $A$. For general type, the standard (monomial) bases of $\RR(\beta)$ and faithful polynomial representations over $\RR(\beta)$ are constructed in \cite{KL1} and \cite{Rou2}, where it is also proved that the centers of the quiver Hecke algebras $\RR(\beta)$ consist of all symmetric elements in its KLR generators $x_1,\cdots,x_n$ and $e(\nu), \nu\in I^\beta$, which is similar to the well-known Bernstein's theorem on the centers of affine Hecke algebras. The representation theory of $\RR(\beta)$ has been well-studied in the literature, see e.g., \cite{BKOP}, \cite{BKM}, \cite{K14}, \cite{KLo}, \cite{KR10,KR11} and the references therein. In contrast to these results, little is known about the structure and representation theory of the cyclotomic quiver Hecke algebra $\RR^\Lam(\beta)$ except the cases of type $A$, type $C$ and some special $\Lam$ (\cite{AP14, AP16c, APS, BK:GradedKL, BKgraded, HM}).

One of the main obstacles for the understanding of $\RR^\Lam(\beta)$ is the lack of an explicit basis or even a closed formula for its graded dimension.
In the case of types $A_\ell^{(1)}$ and $A_\infty$, Brundan and Kleshchev gave in \cite[Theorem 4.20]{BKgraded} a graded dimension formula for $\RR^\Lam(\beta)$ using the enumerative combinatoric of standard tableaux for multi-partitions, and they constructed in \cite{BK:GradedKL} an explicit $K$-algebra isomorphism between $\RR^\Lam(\beta)$ and the block algebra labelled by $\beta$ of the cyclotomic Hecke algebra of type $G(\ell,1,n)$ when $\Lam$ has level $\ell$. In this type $A$ case, Ariki's celebrated categorification work \cite{Ariki:can} was upgraded in \cite{BKgraded} to the $\Z$-graded setting via quiver Hecke algebras.
Based on \cite{BK:GradedKL}, the first author of this paper and Mathas have constructed a graded cellular basis for the cyclotomic quiver Hecke algebra $\RR^\Lam(\beta)$ in these cases. In the case of types $C_\ell^{(1)}$ and $C_\infty$, Ariki, Park and Speyer obtained in \cite{AP16c} and \cite[Theorem 2.5]{APS} a graded dimension formula for $\RR^\Lam(\beta)$ in a similar way as \cite[Theorem 4.20]{BKgraded}. In the case of types $A_{2\ell}^{(2)}$ and $D_{\ell+1}^{(2)}$, S. Oh and E. Park have also obtained in \cite[Theorem 6.3]{OP} (see also \cite{AP14}) a graded dimension formula for the finite quiver Hecke algebra $\RR^{\Lam_0}(\beta)$ using the enumerative combinatoric of standard tableaux for proper Young walls. Both \cite[Theorem 2.5]{APS}, \cite[Theorem 4.20]{BKgraded} and \cite[Theorem 6.3]{OP} rely on the realizations of the Fock space representations of the quantum groups of affine types. Park has given in \cite[Theorem 2.9]{P} an explicit basis of the cyclotomic quiver Hecke algebra corresponding to a minuscule representation of finite type. Recently, Mathas and Tubbenhauer have constructed graded cellular bases for some special affine types, see \cite{MT1}, \cite{MT2}.

In this paper we give a simple and closed formula for the graded dimension of the cyclotomic quiver Hecke algebra $\RR^\Lam(\beta)$ associated to an {\it arbitrary} symmetrizable Cartan matrix $A=(a_{ij})_{i,j\in I}$, where $\Lam\in P^+$ and $\beta\in Q_n^+$. Our new dimension formula is a simple function in terms of the dominant integral weight $\Lam$, simple roots and certain Weyl group elements, and involves no enumerative combinatoric of standard tableaux or Young walls. The following theorem is the first main result of this paper.

\begin{thm}\label{mainthmA}  Let $\beta\in Q_n^+$ and $\nu=(\nu_1,\cdots,\nu_n),\nu'=(\nu'_1,\cdots,\nu'_n)\in I^\beta$. Then $$
\dim_q e(\nu)\RR^\Lam(\beta)e(\nu')=\sum_{\substack{w\in\Sym(\nu,\nu')}}\prod_{t=1}^{n}\Bigl([N^{\Lam}(w,\nu,t)]_{\nu_t}
q_{\nu_t}^{N^{\Lam}(1,\nu,t)-1}\Bigr).
$$
where $e(\nu), e(\nu')$ are the KLR idempotents labelled by $\nu, \nu'$ respectively in the definition of $\RR^\Lam(\beta)$ (Definition \ref{D:QuiverRelations}), $N^{\Lam}(w,\nu,t)$ is an integer given in Definition \ref{keydfn1}, $\Sym(\nu,\nu'):=\{w\in\Sym_n|w\nu=\nu'\}$, $q_{\nu_t}:=q^{d_{\nu_t}}$, $[m]_{\nu_t}$ is the quantum integer introduced in (\ref{quantum1}) and (\ref{quantum2}).
\end{thm}

Since $\{e(\nu)|\nu\in I^\beta\}$ are pairwise orthogonal idempotents whose sum is the identity, we see that $\RR^\Lam(\beta)=\oplus_{\nu,\nu'\in I^\beta}e(\nu)\RR^\Lam(\beta)e(\nu')$ and thus $$
\dim_q \RR^\Lam(\beta)=\sum_{\nu,\nu'\in I^\beta}\dim_q e(\nu)\RR^\Lam(\beta)e(\nu').
$$
The proof of Theorem \ref{mainthmA} relies crucially on Oh-Park's work (\cite[Proposition 3.3]{OP}) which is deduced from Kang-Kashiwara's categorification Theorem. Specializing $q$ to $1$, we get that \begin{equation}\label{dimFormula}
\dim e(\nu)\RR^{\Lambda}(\beta)e(\nu')=\sum\limits_{w\in\Sym(\nu,\nu')}\prod\limits_{t=1}^{n}N^\Lam(w,\nu,t).
\end{equation}
 A priori, those integers $N^{\Lam}(w,\nu,t)$ appeared in the above equality could be negative. Since $\dim e(\nu)\RR^{\Lambda}(\beta)e(\nu')\geq 0$, the summation in the right-hand side of the above equality must be always non-negative. This is surprising as we see no reason why this should be true from only the right-hand side formula itself. Our formula reveals the significance of new numeric invariants, which appear as coefficients, suggesting that their full role is yet to be fully explored. A second simplified (or divided power) version of the dimension formula for $e(\nu)\RR^{\Lambda}(\beta)e(\nu)$ is also obtained in Theorem \ref{mainthm1b}.


Our dimension formula for $\RR^{\Lambda}(\beta)$ depends only on the root system associated to $A$ and the dominant weight $\Lam$, but not on the chosen ground field $K$ and the polynomials $Q_{ij}(u,v)$. This immediately implies that if each $Q_{ij}(u,v)$ is defined over $\Z$ then $\RR^{\Lambda}(\beta)_{\Z}$ is free over $\Z$, and hence $\O\otimes_{\Z}\RR^{\Lambda}(\beta)_{\Z}\cong\RR^{\Lambda}(\beta)_{\O}$ for any commutative ground ring $\O$, which recovers a result in \cite[Proposition 2.4]{APS}, where we use $\RR^{\Lambda}(\beta)_{\O}$ to emphasis the ground ring $\O$ over which the quiver Hecke algebra is defined.

The above dimension formula is new even in the special cases of (affine) type $A$ or (affine) type $C$. By the main results of \cite{BK:GradedKL}, the block algebra labelled by $\beta\in Q_n^+$ of the symmetric group $\Sym_n$ in characteristic $e>0$ and of the Iwahori-Hecke algebra at a primitive $e$th root of unity can be identified with the corresponding cyclotomic quiver Hecke algebra $\RR^{\Lam_0}(\beta)$. Thus Theorem \ref{mainthmA} and (\ref{dimFormula}) give some closed formulae for the dimensions of these block algebras, which is new to the best of our knowledge.
It would be very interesting to relate those integers $N^{\Lam}(w,\nu,t)$ to the Fock space realization of affine quantum groups for general types.

It is well-known that any KLR idempotent $e(\nu)$ in the quiver Hecke algebra $\RR(\beta)$ is nonzero. In contrast, this is in general not the case for the KLR idempotent $e(\nu)$ in the cyclotomic quiver Hecke algebra $\RR^\Lam(\beta)$. In fact, one of the unsolved open problems in the structure and representation theory of $\RR^\Lam(\beta)$ is to determine when the KLR idempotent $e(\nu)$ is nonzero in $\RR^\Lam(\beta)$. As a first application of our new dimension formula Theorem \ref{mainthmA} and (\ref{dimFormula}), we obtain the following second main result of this paper, which gives a simple criterion and thus completely solves the above problem for {\it arbitrary} symmetrizable Cartan matrix.

\begin{thm}\label{mainthmB} Let $\Lam\in P^+$, $\beta\in Q^+$ and $\nu=(\nu_1,\cdots,\nu_n)\in I^\beta$. Then $e(\nu)\neq 0$ in $\RR^\Lam(\beta)$ if and only if $$
\sum\limits_{w\in\Sym(\nu,\nu)}\prod\limits_{t=1}^{n}N^\Lam(w,\nu,t)\neq 0 .
$$
\end{thm}
Using a second version of the dimension formula for $e(\nu)\RR^{\Lambda}(\beta)e(\nu)$ given in Theorem \ref{mainthm1b}, we also obtain in Theorem \ref{mainthmB2} a simplified (or divided power) version of the criterion for $e(\nu)\neq 0$ in $\RR^\Lam(\beta)$.

In a second application of our new dimension formula Theorem \ref{mainthmA} and (\ref{dimFormula}), we prove the following third main result of this paper, which gives a decomposition of $\dim\RR^\Lam(\beta)$ into a sum of some products of $\dim\RR^{\Lam^i}(\beta_i)$ with $\Lam=\sum_i\Lam^i$ and $\beta=\sum_{i}\beta_i$.

\begin{thm} \label{mainthmC}
Suppose $\Lam=\Lam^1+\cdots+\Lam^l$, where $\Lam^i\in P^+$ for each $1\leq i\leq l$. Then
$$\dim \RR^\Lam(\beta)=\sum_{\substack{\beta_1,\cdots,\beta_l\in Q^+\\ \beta=\beta_1+\cdots+\beta_l}}\Bigl(\frac{(|\beta_1|+\cdots+|\beta_l|)!}{|\beta_1|!\cdots|\beta_l|!}\Bigr)^2 \dim \RR^{\Lam^1}(\beta_1)\cdots \dim \RR^{\Lam^l}(\beta_l).
$$
\end{thm}

Our third application of Theorem \ref{mainthmA} is the construction of monomial bases for $\RR^\Lam(\beta)$, which is the starting point of this work. As is well known, constructing monomial bases for the cyclotomic quiver Hecke algebra $\RR^\Lam(\beta)$ is a challenging problem. The first author of this paper and Liang have constructed a monomial basis for the cyclotomic nilHecke algebra in \cite{HuL}. In general, even in the special case of type $A$, no such monomial basis is known at the moment. Our new dimension formula for $\dim\RR^\Lam(\beta)$ gives us a very strong indication that those integers $N^{\Lam}(w,\nu,t)$ might play a key role in the construction of monomial bases of $\RR^\Lam(\beta)$ for general types.

In our fourth main result, we shall construct monomial basis for certain special bi-weight subspace of $\RR^\Lam(\beta)$.
To state the result, we need some notations. We fix $p\in\N$, $\fb:=(b_1,\cdots,b_p)\in\N^p$ and $\nu^1,\cdots,\nu^p\in I$ such that $\nu^i\neq\nu^j$ for any $1\leq i\neq j\leq p$ and $\sum_{i=1}^{p}b_i=n$. We define $$
\wnu=(\wnu_1,\cdots,\wnu_n):=\bigl(\underbrace{\nu^1,\cdots,\nu^1}_{\text{$b_1$ copies}},\cdots,\underbrace{\nu^p,\cdots,\nu^p}_{\text{$b_p$ copies}}\bigr)\in I^\beta ,
$$
where $\beta\in Q_n^+$. Note that each $\mu\in I^\beta$ is in the same $\Sym_n$-orbit as some $\widetilde{\nu}$ of the above form. The following theorem is the fourth main result of this paper. Once again, the theorem is valid for {\it arbitrary} symmetrizable Cartan matrix.
\begin{thm}\label{mainthmD} Let $\mu\in I^\beta$ and $\wnu$ be given as in the last paragraph. Then $$ \text{$e(\wnu)\RR^\Lam({\beta})e(\mu)\neq 0$ if and only if $N^\Lam(\mu,k)>0$ for any $1\leq k\leq n$,} $$
where $N^\Lam(\mu,k)$ is defined as in (\ref{wt}). In that case, fix any reduced expression $w=s_{i_1}\cdots s_{i_t}\in \Sym(\mu,\wnu)$ and define $\psi_w=\psi_{i_1}\cdots\psi_{i_t}$. The following set $$
\Bigl\{\psi_{w }\prod_{k=1}^{n}x_{k}^{r_{k}}e(\mu)\Bigm| w \in\Sym(\mu,\wnu), 0\leq r_{k}<N^\Lam(\mu,k), \forall\,1\leq k\leq n\Bigr\}$$
gives a $K$-basis of $e(\wnu)\RR^\Lam({\beta})e(\mu)$.
\end{thm}
We call the above basis a {\it monomial basis} of $e(\wnu)\RR^\Lam({\beta})e(\mu)$. Applying the anti-isomorphism ``$\ast$'', one can also get a monomial basis for the subspace $e(\mu)\RR^\Lam({\beta})e(\wnu)$. The main difficulty in generalizing the above theorem to arbitrary direct summand
$e(\mu)\RR^\Lam({\beta})e(\nu)$ lies in the fact the integers $N^\Lam(w,\mu,k)$ could be negative. However, we construct the monomial bases for all the direct summands in the $n=3$ case in Subsection 5.3. The construction still indicates the expected monomial bases have some close relationships with those integers $N^\Lam(w,\mu,k)$. We also apply our main results Theorem \ref{mainthmA} and Corollary \ref{maincor1} to give some concrete examples to show that the cyclotomic quiver Hecke algebra $\RR^\Lam(n):=\oplus_{\beta\in Q_n^+}\RR^\Lam(\beta)$ is in general not graded free over its subalgebra $\RR^\Lam(m)$ for $m\leq n$.

%

The content of the paper is organised as follows. In Section 2 we give some preliminary definitions and results on the quantum groups $U_q(\mathfrak{g})$ associated to an arbitrary symmetrizable generalized Cartan matrix $A$, quiver Hecke algebra $\RR(\beta)$ and cyclotomic quiver Hecke algebra $\RR^\Lam(\beta)$ associated to $A, \beta\in Q_n^+$, polynomials $\{Q_{i,j}(u,v)\}$ and $\Lam\in P^+$. In Section 3 we give the proof of our first main result Theorem \ref{mainthmA}. The proof of Theorem \ref{mainthmA} essentially relies on Kang-Kashiwara's categorification of the integral highest weight module $V(\Lam)$ via the category of finite dimensional projective modules over $\RR^\Lam(\beta)$. We give in Theorem \ref{mainthm1b} a second version of the dimension formula for the direct summand $e(\nu)\RR^\Lam(\beta)e(\nu)$. Our second main results Theorem \ref{mainthmB} is proved in Subsection 3.3. In Section 4 we prove several level reduction results in Theorem \ref{mainthmC1} and Corollary \ref{genaralizetion} for the dimension formulae. As a consequence, we obtain in Corollary \ref{maincorC3} a third necessary and sufficient condition for the KLR idempotent $e(\nu)$ to be nonzero in $\RR^\Lam(\beta)$. In Section 5 we apply Theorem \ref{mainthmA} to the construction of monomial bases of $\RR^\Lam(\beta)$. We give the proof of our fourth main result Theorem \ref{mainthmD} in this section. We first construct a monomial bases of $e(\wnu)\RR^\Lam(\beta)e(\wnu)$ in Subsection 5.1. Then we construct a monomial bases of $e(\wnu)\RR^\Lam(\beta)e(\mu)$ for arbitrary $\mu$ in Subsection 5.2. Using the results obtained in Subsections 5.1, 5.2, we are able to construct in Subsection 5.3 a monomial basis for arbitrary direct summand $e(\mu)\RR^\Lam(\beta)e(\nu)$ of $\RR^\Lam(\beta)$ in the case $n=3$. Finally we give in Subsection 5.4 some concrete examples to show that the cyclotomic quiver Hecke algebra $\RR^\Lam(n):=\oplus_{\beta\in Q_n^+}\RR^\Lam(\beta)$ is in general not graded free over its subalgebra $\RR^\Lam(m)$ for $m<n$.

\bigskip
\centerline{Acknowledgements}
\bigskip

The research was supported by the National Natural Science Foundation of China (No. 12171029) and Beijing Natural Science Foundation (No.1232017). Both authors are grateful to the referee for his/her careful reading and invaluable suggestion and comments.
\bigskip

\section{Preliminary}

In this section we shall recall some basic knowledge about the quantum groups and (cyclotomic) quiver Hecke algebras.

Let $A:=(a_{ij})_{i,j\in I}$ be a symmetrizable generalized Cartan matrix. Let $\{d_i\in\Z_{>0}|i\in I\}$ be a family of positive integers such that $(d_ia_{ij})_{i,j\in I}$ is symmetric.
Let $(P,\Pi,\Pi^\vee)$ be a realization of $A$ and $\mathfrak{g}$ be the corresponding Kac-Moody Lie algebra (\cite{Kac}). In other words, $P$ is a free abelian group called the weight lattice, $\Pi=\{\alpha_i|i\in I\}$ is the set of simple roots, $\Pi^\vee=\{h_i|i\in I\}\subset P^\vee:=\Hom_\Z(P,\Z)$ is the set of simple coroots, $\<\alpha_j,h_i\>=a_{ij}$, $\forall\,i,j\in I$, and $\Pi, \Pi^\vee$ are linearly independent sets.

There is a symmetric bilinear pairing $(-|-)$ on $P$ satisfying $$
(\alpha_j|\alpha_i)=d_ia_{ij},\quad (\Lam|\alpha_i)=d_i\<\Lam,h_i\>, \,\,\forall\,\Lam\in P.
$$
In particular, $d_i=(\alpha_i|\alpha_i)/2$. We denote by $P^+=\{\Lam\in P|\<\Lam,h_i\>\geq 0,\forall\,i\in I\}$ the set of dominant integral weights. For each $i\in I$, let $\Lam_i$ be the $i$th fundamental weight, i.e., $\<\Lam_i,h_j\>=\delta_{ij}, \forall\,j\in I$. Then each $\Lam\in P^+$ can be written as $\Lam=\sum_{i\in I}k_i\Lam_i$, and we call $\ell(\Lam):=\sum k_i$ the level of $\Lam$.

Let $q$ be an indeterminate. For any $k\in I$, we set $q_k:=q^{d_k}=q^{(\alpha_k|\alpha_k)/2}$. For any $m\in\Z$, we define \begin{equation}\label{quantum1}
[m]_{k}:=\frac{q_k^m-q_k^{-m}}{q_k-q_{k}^{-1}}.
\end{equation}
For any $m,n\in\N$ with $m\geq n$, we define \begin{equation}\label{quantum2}
[m]^{!}_{k}:=\prod_{t=1}^{m}[t]_k,\,\,\biggl[\begin{matrix}m\\ n\end{matrix}\biggr]_k:=\frac{[m]_k^{!}}{[m-n]_k^{!}[n]_k^{!}} .
\end{equation}
If $d_k=1$ for any $k\in I$, then we shall omit the subscript $k$ and write $[m]$ instead of $[m]_{k}$.

\begin{dfn}\label{qgrp} The quantum group (or quantized enveloping algebra) $U_q(\fg)$ (\cite{Lu}) associated with $(A,P,\Pi,\Pi^\vee)$ is the associative algebra over $\Q(q)$ with $1$ generated by $e_i,f_i$ ($i\in I$) and $q^h$ ($h\in P^\vee$) satisfying the following relations: $$\begin{aligned}
(1)\,\,& q^0=1,\,\, q^h q^{h'}=q^{h+h'},\,\,\forall\, h,h'\in P^\vee;\\
(2)\,\,& q^he_iq^{-h}=q^{\<\alpha_i,h\>}e_i,\,q^hf_iq^{-h}=q^{-\<\alpha_i,h\>}f_i,\,\,\forall\,h\in P^\vee, i\in I;\\
(3)\,\,& e_if_j-f_je_i=\delta_{ij}\frac{K_i-K_i^{-1}}{q_i-q_i^{-1}},\,\,\text{where $K_i=q_i^{h_i}$};\\
(4)\,\,& \sum_{k=0}^{1-a_{ij}}(-1)^k\biggl[\begin{matrix}1-a_{ij}\\ k\end{matrix}\biggr]_ie_i^{1-a_{ij}-k}e_je_i^k=0,\,\,\forall\,i\neq j;\\
(5)\,\,& \sum_{k=0}^{1-a_{ij}}(-1)^k\biggl[\begin{matrix}1-a_{ij}\\ k\end{matrix}\biggr]_if_i^{1-a_{ij}-k}f_jf_i^k=0,,\,\forall\,i\neq j.
\end{aligned}
$$
\end{dfn}

We set $Q:=\bigoplus_{i\in I}\Z\alpha_i$, and call it the root lattice. Set $Q^+:=\bigoplus_{i\in I}\N\alpha_i$, and call it the positive root lattice. For each $\beta=\sum_{i\in I}k_i\alpha_i\in Q^+$, we define $|\beta|:=\sum_{i\in I}k_i$. For each $n\in\N$, we set $Q_n^+:=\{\beta\in Q^+||\beta|=n\}$.

Let $u,v$ be two indeterminates. For any $i,j\in I$, let $Q_{i,j}(u,v)\in K[u.v]$ be a polynomial of the form $$
Q_{i,j}(u,v)=\begin{cases} \sum_{p(\alpha_i|\alpha_i)+q(\alpha_j|\alpha_j)+2(\alpha_i|\alpha_j)=0}t_{i,j;p,q}u^pv^q, &\text{if $i\neq j$;}\\
0, &\text{if $i=j$,}
\end{cases}
$$
where $t_{i,j;p,q}\in K$ are such that $t_{i,j;-a_{ij},0}\in K^\times$, and they satisfy that $Q_{i,j}(u,v)=Q_{j,i}(v,u)$, $\forall\,i,j\in I$.
In particular, if we regard $Q_{i,j}(u,v)$ as a polynomial on $u$, then the highest degree of $u$ in $Q_{i,j}(u,v)$ is $-a_{ij}$ with leading coefficient $t_{i,j;-a_{ij},0}\in K^\times$.

Let $I^n:=\{\nu=(\nu_1,\cdots,\nu_n)|\nu_i\in I,\forall\,1\leq i\leq n\}$. For any $\beta\in Q_n^+$, we define $$
I^\beta=\biggl\{\nu=(\nu_1,\cdots,\nu_n)\in I^n\biggm|\sum_{i=1}^n\alpha_{\nu_i}=\beta\biggr\}.
$$

Let $\Sym_n$ be the symmetric group on $\{1,2,\cdots,n\}$. Then $\Sym_n$ acts on $I^n$ from the left-hand side by places permutation. That is, for any $w\in\Sym_n$, $\nu=(\nu_1,\cdots,\nu_n)$, $$
w\nu=w(\nu_1,\cdots,\nu_n):=(\nu_{w^{-1}(1)},\cdots,\nu_{w^{-1}(n)}) .
$$
One can also consider the action of $\Sym_n$ on $I^n$ from the right-hand side, then we have $$
\nu w=(\nu_1,\cdots,\nu_n)w:=(\nu_{w(1)},\cdots,\nu_{w(n)}) .
$$
In particular, $w\nu=\nu w^{-1}$.

\begin{dfn}\label{D:QuiverRelations}
  Let $K$ be a field. Let $n\in\N$ and $\beta\in Q_n^{+}$. The quiver Hecke (or KLR) algebra $\RR(\beta)$ associated with polynomial $(Q_{i,j}(u,v))_{i,j\in I}$ and $\beta\in Q_n^+$ is the unital associative $K$-algebra with generators
  $$\{\psi_1,\dots,\psi_{n-1}\} \cup   \{x_1,\dots, x_n \} \cup \{e(\nu)|\nu\in I^\beta\} $$
  and relations
  \bgroup
      \setlength{\abovedisplayskip}{1pt}
      \setlength{\belowdisplayskip}{1pt}
  \begin{align*}
         e(\nu) e(\nu') &= \delta_{\nu\nu'} e(\nu),
    & \sum_{\nu\in I^\beta}e(\nu)=1,& & &\\
    x_r e(\nu) &= e(\nu) x_r,
    &\psi_r e(\nu)&= e(s_r\nu) \psi_r,
    &x_r x_s &= x_s x_r,
  \end{align*}
  \begin{align*}
    \psi_r x_{r+1} e(\nu)&=(x_r\psi_r+\delta_{\nu_r\nu_{r+1}})e(\nu),&
    x_{r+1}\psi_re(\nu)&=(\psi_r x_r+\delta_{\nu_r\nu_{r+1}})e(\nu),\\
    \psi_r x_s  &= x_s \psi_r,&&\text{if }s \neq r,r+1,\\
    \psi_r \psi_s &= \psi_s \psi_r,&&\text{if }|r-s|>1,\notag
\end{align*}
\begin{align*}
  \psi_r^2e(\nu) &= Q_{\nu_r,\nu_{r+1}}(x_r,x_{r+1})e(\nu),\\
\psi_{r+1}\psi_{r} \psi_{r+1} e(\nu)-\psi_{r}\psi_{r+1} \psi_{r} e(\nu) &=\delta_{\nu_r \nu_{r+2}}\frac{Q_{\nu_r,\nu_{r+1}}(x_r,x_{r+1})-Q_{\nu_r,\nu_{r+1}}(x_{r+2},x_{r+1})}{x_r-x_{r+2}}e(\nu),
\end{align*}
\egroup
for $\nu,\nu'\in I^\beta$ and all admissible $r$ and $s$.
\end{dfn}

For $\Lam\in P^+$, $i\in I$, we define $$
a^\Lam_i(x)=x^{\<\Lam,h_i\>}.
$$

\begin{dfn} The cyclotomic quiver Hecke (or cyclotomic KLR) algebra $\RR^\Lam(\beta)$ associated with the polynomial $(Q_{i,j}(u,v))_{i,j\in I}$, $\beta\in Q_n^+$ and $\Lam\in P^+$ is defined to be the quotient of $\RR(\beta)$ by the two-sided ideal of $\RR(\beta)$ generated by $a_{\nu_1}^\Lam(x_1)e(\nu)$, $\nu\in I^\beta$.
\end{dfn}

The idempotents $e(\mu)\in\RR(\beta)$ and $e(\nu)\in\RR^\Lam(\beta)$ will be called the KLR idempotents of $\RR(\beta)$ and $\RR^\Lam(\beta)$ respectively. The algebra $\RR(\beta)$ is $\Z$-graded with its grading structure given by  $$
\deg e(\nu)=0,\,\,\deg(x_ke(\nu)):=(\alpha_{\nu_k}|\alpha_{\nu_k}),\,\,\deg(\psi_ke(\nu)):=-(\alpha_{\nu_k}|\alpha_{\nu_{k+1}}) .
$$
Inheriting the $\Z$-grading from $\RR(\beta)$, the cyclotomic quiver Hecke algebra $\RR^\Lam(\beta)$ is $\Z$-graded too. There is a unique $K$-algebra anti-isomorphism ``$\ast$'' of $\RR^\Lam(\beta)$ which is defined on its KLR generators by $$
e(\nu)^\ast=e(\nu),\,\,\psi_r^{\ast}:=\psi_r,\,\,x_s^\ast:=x_s,\,\,\,\forall\, \nu\in I^\beta, 1\leq r<n, 1\leq s\leq n .
$$

We use $q$ to denote the grading shift functor on $\Mod(\RR^\Lam(\beta))$. That means $$(qM)_j=M_{j-1},
$$ for any $M=\oplus_{\substack{j\in \Z}}M_j\in\Mod(\RR^\Lam(\beta))$. Then the Grothendieck group $[\Mod(\RR^\Lam(\beta))]$ becomes a $\Z[q,q^{-1}]$-module, where $q[M]=[qM]$ for $M\in \Mod(\RR^\Lam(\beta))$.
Let $\beta\in Q_n^+$ and $i\in I$, we set $$
e(\beta,i):=\sum_{\nu=(\nu_1,\cdots,\nu_n)\in I^{\beta}}e(\nu_1,\cdots,\nu_n,i).
$$
Kang and Kashiwara have introduced restriction functors and induction functors in \cite{KK} as follows: $$\begin{aligned}
E_i^\Lam:\, \Mod(\RR^\Lam(\beta+\alpha_i))&\rightarrow \Mod(\RR^\Lam(\beta)),\\
N&\mapsto e(\beta,i)N=e(\beta,i)\RR^\Lam(\beta+\alpha_i)\otimes_{\RR^{\Lam}(\beta+\alpha_i)}N,\\
F_i^\Lam:\, \Mod(\RR^\Lam(\beta))&\rightarrow \Mod(\RR^\Lam(\beta+\alpha_i)),\\
M&\mapsto \RR^\Lam(\beta+\alpha_i)e(\beta,i)\otimes_{\RR^{\Lam}(\beta)}M .
\end{aligned}
$$
Let $\Proj(\RR^\Lam(\beta))$ be the category of finite dimensional projective $\RR^\Lam(\beta)$-modules and $K\bigl(\Proj(\RR^\Lam(\beta))\bigr)$ its Grothendieck group. Let $\rm{K}_i$ be the endomorphism of $K\bigl(\Proj(\RR^\Lam(\beta))\bigr)$ given by multiplication of $q_i^{\<\Lam-\beta,h_i\>}$.
Let ${\rm E}_i:=q_i^{1-\<\Lam-\beta,h_i\>}[E_i^\Lam]$, ${\rm F}_i:=[F_i^\Lam]$, where $[E_i^\Lam]: K\bigl(\Proj(\RR^\Lam(\beta+\alpha_i))\bigr)\rightarrow K\bigl(\Proj(\RR^\Lam(\beta))\bigr)$ and $[F_i^\Lam]: K\bigl(\Proj(\RR^\Lam(\beta))\bigr)\rightarrow K\bigl(\Proj(\RR^\Lam(\beta+\alpha_i))\bigr)$ are the naturally induced map on the Grothendieck groups. Then by \cite[Lemma 6.1]{KK}, \begin{equation}\label{effe}
{\rm E}_i{\rm F}_j-{\rm F}_j{\rm E}_i=\delta_{ij}\frac{{\rm K}_i-{\rm K}_i^{-1}}{q_i-q_i^{-1}} .
\end{equation}

Let $U_{\Z[q,q^{-1}]}(\mathfrak{g})$ be the Lusztg's $\Z[q,q^{-1}]$-form of the quantum group $U_q(\mathfrak{g})$. Let $v_\Lam$ be a fixed highest weight vector of the irreducible highest weight $U_q(\mathfrak{g})$-module $V(\Lam)$. Set $V_{\Z[q,q^{-1}]}(\Lam):=U_{\Z[q,q^{-1}]}(\mathfrak{g})v_\Lam$.

\begin{thm}\text{(\cite{KK})} For each $\Lam\in P^+$, there is an $U_{\Z[q,q^{-1}]}(\mathfrak{g})$-module isomorphism: $K\bigl(\Proj\RR^\Lam\bigr)\cong V_{\Z[q,q^{-1}]}(\Lam)$.
\end{thm}

For each $1\leq i<n$, we define $s_i:=(i,i+1)$. Then $s_1,\cdots,s_{n-1}$ generates $\Sym_n$. A word $w=s_{i_{1}}s_{i_{2}}\ldots s_{i_{k}}$ for $w\in \Sym_{n}$ is called a reduced expression of $w$ if $k$ is minimal; in this case we say that $w$ has length $k$ and we write $\ell(w)=k$. We use ``$\leq$'' to denote the Bruhat partial order on $\Sym_n$. That is, for any $x,y\in\Sym_n$, $x\leq y$ if and only if $x=s_{i_{j_1}}\cdots s_{i_{j_t}}$ for some reduced expression
$y=s_{i_1}\cdots s_{i_m}$ of $y$ and some integers $1\leq t\leq m$, $1\leq j_1<\cdots< j_t\leq m$. If $x\leq y$ and $x\neq y$ then we write $x<y$.

\begin{lem}\label{deg1} Let $w\in\Sym_n$ and $\nu=(\nu_1,\cdots,\nu_n)\in I^n$. We fix a reduced expression $s_{r_1}\cdots s_{r_k}$ of $w$, and define $\psi_w:=\psi_{r_1}\cdots\psi_{r_k}$. Then $$
\deg\psi_we(\nu)=-\sum_{t=1}^{n}\sum_{\substack{1\leq i<t\\ w(i)>w(t)}}(\alpha_{\nu_i}|\alpha_{\nu_t}).
$$
In particular, $\deg\psi_we(\nu)$ is independent of the choice of the reduced expression $s_{r_1}\cdots s_{r_k}$ of $w$.
\end{lem}

\begin{proof} We define $n(w)=\{(i,j)|1\leq i<j\leq n, w(i)>w(j)\}$. To prove the lemma we make induction on $\ell(w)$. If $\ell(w)=1$, the lemma follows from
the definition of $\deg\psi_r$.

Now suppose $\ell(w)>1$. Then we can always choose $1\leq t<n$ such that $s_tw<w$. In particular, $\ell(s_{t}w)+1=\ell(w)$. In this case it is easy to check $$
n(w)=n(s_{t}w)\cup\{(w^{-1}(t),w^{-1}(t+1))\}. $$

Therefore, we have $$\begin{aligned}
\deg(\psi_{w}e(\nu))&=\deg(\psi_{s_{t}}e(s_{t}w\,\nu))+\deg(\psi_{s_{t}w}e(\nu))\\
&=\deg(\psi_{s_{t}}e\bigl(\nu_{w^{-1}(1)},\cdots,\nu_{w^{-1}(t+1)},\nu_{w^{-1}(t)},\cdots,\nu_{w^{-1}(n)})\bigr)\\
&\qquad\qquad -\sum_{\substack{i<j\\  s_{t}w(i)>s_{t}w(j)}}(\alpha_{\nu_{j}}\,|\alpha_{\nu_{i}})\quad \text{(by induction hypothesis)}\\
&=-(\alpha_{\nu_{w^{-1}(t)}}\,|\alpha_{\nu_{w^{-1}(t+1))}})-\sum_{\substack{i<j \\ s_{t}w(i)>s_{t}w(j)}}(\alpha_{\nu_{j}}\,|\alpha_{\nu_{i}})\\
&=-\sum_{\substack{i<j \\w(i)>w(j)}}(\alpha_{\nu_{j}}\,|\alpha_{\nu_{i}}).\end{aligned} $$
This completes the proof of the lemma.
\end{proof}

\bigskip

\section{Graded dimensions of cyclotomic quiver Hecke algebras}

In this section we shall first give a proof of our first main result Theorem \ref{mainthmA}. That is, to give a closed formula for the graded dimension of the cyclotomic quiver Hecke algebra $\RR^\Lam(\beta)$. Then, as an application of Theorem \ref{mainthmA}, we shall give two criteria for the KLR idempotent  $e(\nu)$ to be nonzero in $\RR^\Lam(\beta)$. In particular, we shall give the proof of our second main result Theorem \ref{mainthmB} of this paper.

\subsection{A graded dimension formula for $\RR^\Lam(\beta)$}

Since $\{e(\nu)|\nu\in I^\beta\}$ are pairwise orthogonal idempotents in $\RR^\Lam(\beta)$ which sum to $1$, we have $$\RR^\Lam(\beta)=\oplus_{\mu,\nu\in I^\beta}e(\mu)\RR^\Lam(\beta)e(\nu).$$
Thus to give the graded dimension formula for $\RR^\Lam(\beta)$, it suffices to give the graded dimension formula for each $e(\mu)\RR^\Lam(\beta)e(\nu)$, where $\mu,\nu\in I^\beta$.

For $\Lam\in P^+$, $\beta\in Q^+$, we define $$
\df(\Lam,\beta):=(\Lam |\beta)-\frac{1}{2}(\beta |\beta) .
$$

\begin{lem}\label{identity1} Let $\Lam\in P^+$, $\beta\in Q^+$. Then for any $\alpha_i\in\Pi$, we have $$
\df(\Lam,\beta)-\df(\Lam,\beta-\alpha_i)=d_i\(1+\<\Lam-\beta,h_i\>\).
$$
\end{lem}

\begin{proof} By definition, $d_i=(\alpha_i |\alpha_i)/2$. It follows that $$\begin{aligned}
&\quad\,\df(\Lam,\beta)-\df(\Lam,\beta-\alpha_i)\\
&=(\Lam |\alpha_i)-(\beta |\alpha_i)+\frac{1}{2}(\alpha_i|\alpha_i)=d_i\(1+\<\Lam-\beta,h_i\>\) .
\end{aligned}
$$
This proves the lemma.
\end{proof}

\begin{dfn}\label{keydfn1} For any $w\in\Sym_n$, $t\in\{1,2,\cdots,n\}$, we define $$
J_w^{<t}:=\{1\leq j<t|w(j)<w(t)\} .
$$
Let $\Lam\in P^+$. For any $\nu=(\nu_1,\cdots,\nu_n)\in I^n$ and $1\leq t\leq n$, we define  \begin{equation}\label{Ndef}
N^\Lam(w,\nu,t):=\<\Lam-\sum_{j\in J_w^{<t}}\alpha_{\nu_j}, h_{\nu_t}\>.
\end{equation}
\end{dfn}

For any $\nu,\nu'\in I^n$, we define
$$
\Sym(\nu,\nu'):=\big\{w\in\Sym_n | w\nu =\nu' \big\} .
$$

\begin{lem}\label{eqa1} Let $\nu,\nu'\in I^n$. For any $w\in\Sym(\nu,\nu')$ and $1\leq t\leq n$, we have that $$
N^\Lam(w,\nu,t)=\<\Lam-\sum_{\substack{1\leq j<w(t),\\ j\in \{w (1),\cdots\, ,w (t-1)\}}}\alpha_{\nu'_{j}},h_{\nu_t}\>.$$
\end{lem}

\begin{proof} For any $1\leq i<w(t)$ with $i\in \{w (1),\cdots\, ,w (t-1)\}$, we can find a unique $j\in J_w^{<t}$ such that $i=w(j)$ and hence $\nu'_i=\nu'_{w(j)}=\nu_j$ because $w\in\Sym(\nu,\nu')$.
The lemma follows at once.
\end{proof}

Let $M$ be a finite dimensional $\Z$-graded $K$-linear space. For each $k\in\Z$, we use $M_k$ to denote its degree $k$ homogeneous component. The graded dimension of $M$ is defined by $$
\dim_q M:=\sum_{k\in\Z}(\dim M_k)q^k .
$$

By the definitions given in the paragraph above (\ref{effe}), we have $$
{\rm F}_i[\RR^\Lam(\beta)]=[\RR^\Lam(\beta+\alpha_i)e(\beta,i)],\quad {\rm E}_i[\RR^\Lam(\beta+\alpha_i)]=q_i^{1-\<\Lam-\beta,h_i\>}[e(\beta,i)\RR^\Lam(\beta+\alpha_i)].
$$
As a result, Oh and Park deduced the following proposition in \cite[Proposition 3.3]{OP} which plays a central role in our main result.

\begin{prop}[\text{\cite[Proposition 3.3]{OP}}]\label{op1} Let $\Lam\in P^+$, $v_\Lam\in V(\Lam)$ be a highest weight vector in $V(\Lam)$ of weight $\Lam$. Let $\beta\in Q^+$ and $\nu=(\nu_1,\cdots,\nu_n),\nu'=(\nu'_1,\cdots,\nu'_n)\in I^\beta$. Then $$
e_{\nu_1}\cdots e_{\nu_n}f_{\nu'_n}\cdots f_{\nu'_1}v_{\Lam}=q^{-\df(\Lam,\beta)}\bigl(\dim_q e(\nu)\RR^\Lam(\beta)e(\nu')\bigr)v_\Lam .
$$
\end{prop}

For each monomial of the form $f_{j_1}\cdots f_{j_n}$, we use the notation $f_{j_1}\cdots \widehat{f_{j_k}}\cdots f_{j_n}$ to denote the monomial obtained by removing $f_{j_k}$ from the monomial  $f_{j_1}\cdots f_{j_n}$. That is, $$
f_{j_1}\cdots \widehat{f_{j_k}}\cdots f_{j_n}:=f_{j_1}\cdots f_{j_{k-1}}f_{j_{k+1}}\cdots f_{j_n}.
$$
Similarly, for any $\nu=(\nu_1,\cdots,\nu_n)\in I^\beta$, we define $$
(\nu_1,\cdots,\widehat{\nu_k},\cdots,\nu_n):=(\nu_1,\cdots,\nu_{k-1},\nu_{k+1},\cdots,\nu_n)\in I^{\beta-\alpha_{\nu_k}}.
$$

\medskip
\noindent
{\bf Proof of Theorem \ref{mainthmA}}: We claim that $$\begin{aligned}
&\quad\,\dim_{q}e(\nu)\RR^{\Lam}(\beta)e(\nu')\\
&=\sum_{\substack{1\leq k_1,\cdots,k_n\leq n\\ \nu_i=\nu'_{k_i}, \forall\,1\leq i\leq n\\ k_a\neq k_b,\forall\,1\leq a\neq b\leq n}}\prod_{t=1}^{n}\Biggl(
\Bigl[\bigl(\Lam-\sum\limits_{\substack{1\leq i<k_t\\ i\neq k_s,\forall\,t\leq s\leq n}}\alpha_{\nu'_i}\bigr)(h_{\nu_t})\Bigr]_{\nu_t}q_{\nu_t}^{N^{\Lam}(1,\nu,t)-1}\Biggr)
\end{aligned}
$$

We use induction on $|\beta|$. Suppose that the claim holds for any $\beta\in Q_{n-1}^+$. Now we assume $\beta\in Q_n^+$. Applying Proposition \ref{op1}, we get that $$\begin{aligned}
&\quad\, \Bigl(\dim_{q}e(\nu)\RR^{\Lam}(\beta)e(\nu')\Bigl)v_{\Lam}\\
&=q^{\df(\Lam,\beta)}e_{\nu_1}\cdots\, e_{\nu_n}f_{\nu'_n} \cdots\, f_{\nu'_1}v_{\Lam}\\
&=\sum_{\substack{1\leq k_n\leq n\\ \nu_n=\nu'_{k_n}}}{q^{\df(\Lam,\beta)}}\Bigl[(\Lam-\sum\limits_{i=1}^{k_{n}-1}\alpha_{\nu'_i})(h_{\nu_n})\Bigr]_{\nu_n}e_{\nu_{1}}\cdots\, e_{\nu_{n-1}}f_{\nu'_n} \cdots\, \widehat{f_{\nu'_{k_n}}} \\
&\qquad\qquad \times\cdots\times f_{\nu'_1}v_{\Lam}\qquad\qquad\text{(by (\ref{effe}) and Definition \ref{qgrp} (2),(3))}\\
&=\sum_{\substack{1\leq k_n\leq n\\ \nu_n=\nu'_{k_n}}} q^{\df(\Lam,\beta)-\df(\Lam,\beta-\alpha_{\nu_n})}
\Bigl[(\Lam-\sum\limits_{i=1}^{k_{n}-1}\alpha_{\nu'_i})(h_{\nu_n})\Bigr]_{\nu_n}\\
&\qquad\quad\times \dim_q e(\nu_1,\cdots,\nu_{n-1})\RR^\Lam(\beta-\alpha_{\nu_n})e(\nu'_1,\cdots,\widehat{\nu'_{k_n}},\cdots,\nu'_n)v_{\Lam}\quad\text{(by Proposition \ref{op1})}\\
&=\sum_{\substack{1\leq k_n\leq n\\ \nu_n=\nu'_{k_n}}}{q_{\nu_{n}}^{1+(\Lambda-\beta)(h_{\nu_{n}})}}\Bigl[(\Lam-\sum\limits_{i=1}^{k_{n}-1}\alpha_{\nu'_i})(h_{\nu_n})\Bigr]_{\nu_n}
\\
&\qquad\quad\times\Bigl(\dim_q e(\nu_1,\cdots,\nu_{n-1})\RR^{\Lam}(\beta-\alpha_{\nu_n})e(\nu'_1,\cdots,\widehat{\nu'_{k_n}} \cdots\, \nu'_n)\Bigl)v_{\Lam}\qquad\text{(by Lemma \ref{identity1})}.
\end{aligned} $$
It follows that  \begin{equation}\label{ind}\begin{aligned}
\dim_{q}e(\nu)\RR^{\Lam}(\beta)e(\nu')
&=\sum_{\substack{1\leq k_n\leq n\\ \nu_n=\nu'_{k_n}}}{q_{\nu_{n}}^{1+(\Lambda-\beta)(h_{\nu_{n}})}}\Bigl[(\Lam-\sum\limits_{i=1}^{k_{n}-1}\alpha_{\nu'_i})(h_{\nu_n})\Bigr]_{\nu_n}\\
&\qquad\qquad\times\dim_q e(\nu_1,\cdots,\nu_{n-1})\RR^{\Lam}(\beta-\alpha_{\nu_n})e(\nu'_1,\cdots,\widehat{\nu'_{k_n}} \cdots\, \nu'_n).
\end{aligned}\end{equation}

We define $\tilde{\nu}'=(\tilde{\nu}'_1,\cdots,\tilde{\nu}'_{n-1}):=(\nu'_1,\cdots,\widehat{\nu'_{k_n}},\cdots,\nu'_n)$. Applying induction hypothesis, we can deduce that $$
\begin{aligned}
&\quad\, \Bigl(\dim_q e(\nu_1,\cdots\,\nu_{n-1})\RR^{\Lambda}(\beta-\alpha_{\nu_n})e(\nu'_1,\cdots,\widehat{\nu'_{k_n}},\cdots,\nu'_n)\Bigr)v_{\Lambda}\\
&=\Bigl(\dim_q e(\nu_1,\cdots\,\nu_{n-1})\RR^{\Lambda}(\beta-\alpha_{\nu_n})e(\tilde{\nu}'_1,\cdots,\tilde{\nu}'_{n-1})\Bigl)v_{\Lambda}\\
&=\sum_{\substack{1\leq \tilde{k}_1,\cdots,\tilde{k}_{n-1}\leq n-1\\ \nu_i=\tilde{\nu}'_{\tilde{k}_i}, \forall\,1\leq i\leq n-1\\ \tilde{k}_a\neq \tilde{k}_b,\forall\,a\neq b}}\prod_{t=1}^{n-1}
\Bigl(\Bigl[\bigl(\Lam-\sum\limits_{\substack{1\leq i<\tilde{k}_t\\ i\neq \tilde{k}_s,\forall\,t\leq s\leq n-1}}\alpha_{\tilde{\nu}'_i}\bigr)(h_{\nu_t})\Bigl]_{\nu_t}q_{\nu_t}^{N^{\Lam}(1,\nu,t)-1}\Bigr)
v_{\Lambda}.
\end{aligned}
$$
Note that the $(n-1)$-tuple in the summation is a permutation of $\{1,2,\cdots,n-1\}$. For any given integer $1\leq k_n\leq n$, there is an associated natural bijection $\pi_{k_n}$ from the set $$
\Bigl\{(k_1,\cdots,k_{n-1})\Bigm|\begin{matrix}\text{$1\leq k_1,\cdots,k_{n-1}\leq n$, $\nu_i=\nu'_{k_i},\forall\,1\leq i\leq n-1$}\\ \text{$k_n\neq k_a\neq k_b, \forall\,1\leq a\neq b<n$}
\end{matrix}\Bigr\}
$$
onto the set $$
\Bigl\{(\tilde{k}_1,\cdots,\tilde{k}_{n-1})\Bigm|\begin{matrix}\text{$1\leq \tilde{k}_1,\cdots,\tilde{k}_{n-1}\leq n-1$, $\nu_i=\tilde{\nu}'_{\tilde{k}_i},\forall\,1\leq i\leq n-1$}\\ \text{$\tilde{k}_a\neq \tilde{k}_b, \forall\,1\leq a\neq b<n$}
\end{matrix}\Bigr\}
$$
which is defined by $$
\pi_{k_n}(k_1,\cdots,k_{n-1})=(\tilde{k}_1,\cdots,\tilde{k}_{n-1}),\quad
\tilde{k}_j:=\begin{cases}k_j, &\text{if $k_j<k_n$;}\\ k_j-1, &\text{if $k_j>k_n$.}\end{cases}\,\,\forall\,1\leq j\leq n-1.
$$
With this bijection $\pi_{k_n}$ in mind, we can deduce from the above calculation that $$
\begin{aligned}
&\quad\, \Bigl(\dim_q e(\nu_1,\cdots\,\nu_{n-1})\RR^{\Lambda}(\beta-\alpha_{\nu_n})e(\nu'_1,\cdots,\widehat{\nu'_{k_n}},\cdots,\nu'_n)\Bigl)v_{\Lambda}\\
&=\sum_{\substack{1\leq k_1,\cdots,k_{n-1}\leq n\\ \nu_i={\nu}'_{k_i}, \forall\,1\leq i\leq n-1\\ k_n\neq k_a\neq k_b,\forall\,1\leq a\neq b<n}}\prod_{t=1}^{n-1}
\Bigl(\Bigl[\bigl(\Lam-\sum\limits_{\substack{1\leq i<k_t\\ i\neq k_s,\forall\,t\leq s\leq n-1}}\alpha_{\nu'_i}\bigr)(h_{\nu_t})\Bigr]_{\nu_t}q_{\nu_t}^{N^{\Lam}(1,\nu,t)-1}\Bigr)
v_{\Lambda}.
\end{aligned}
$$
Combining this with the equality (\ref{ind}), we prove our claim.

Finally, $\{k_1,\cdots,k_n\}$ is a permutation of $\{1,\cdots,n\}$ and $\nu_i={\nu}'_{k_i}, \forall\,1\leq i\leq n$ mean that there exists $w\in\Sym(\nu,\nu')$ such that $k_j=w(j)$, $\forall\,1\leq j\leq n$. Then it is clear that the theorem follows from our above claim and Lemma \ref{eqa1}.\qed

\begin{cor}\label{maincor1} Let $\beta\in Q^+$ and $\nu=(\nu_1,\cdots,\nu_n),\nu'=(\nu'_1,\cdots,\nu'_n)\in I^\beta$. Then
$\dim e(\nu)\RR^{\Lambda}(\beta)e(\nu')=\sum\limits_{w\in\Sym(\nu,\nu')}\prod\limits_{t=1}^{n}N^\Lam(w,\nu,t)$.
\end{cor}

\begin{proof} We evaluate the formula in Theorem \ref{mainthmA} at $q=1$ by applying the L'Hospital rule. The corollary follows.
\end{proof}

Let $\nu,\nu'\in I^\beta$. We fix an element $w \in\Sym(\nu,\nu')$. Applying Lemma \ref{deg1}, we can get $$
\prod_{t=1}^{n}q_{\nu_t}^{N^{\Lam}(1,\nu,t)-1}
=q^{\deg\psi_we(\nu)}\prod_{t=1}^{n}q_{\nu_t}^{N^{\Lam}(w,\nu,t)-1}
$$
It follows that $$\begin{aligned}
&\quad\,\dim_q e(\nu)\RR^\Lam(\beta)e(\nu')=\sum_{\substack{w\in\Sym(\nu,\nu')}}\prod_{t=1}^{n}\Bigl([N^{\Lam}(w,\nu,t)]_{\nu_t}
q_{\nu_t}^{N^{\Lam}(1,\nu,t)-1}\Bigr)\\
&=\sum_{\substack{w\in\Sym(\nu,\nu')}}\prod_{\substack{1\leq t\leq n\\ N^{\Lam}(w,\nu,t)\neq 0,\forall\,t}}\Bigl([N^{\Lam}(w,\nu,t)]_{\nu_t}
q_{\nu_t}^{N^{\Lam}(1,\nu,t)-1}\Bigr)\\
&=\sum_{\substack{w\in\Sym(\nu,\nu')}}q^{\deg(\psi_we(\nu))}\prod_{\substack{1\leq t\leq n\\ N^{\Lam}(w,\nu,t)\neq 0,\forall\,t}}\Bigl([N^{\Lam}(w,\nu,t)]_{\nu_t}
q_{\nu_t}^{N^{\Lam}(w,\nu,t)-1}\Bigr).
\end{aligned}
$$
If $N^{\Lam}(w,\nu,t)>0$, then \begin{equation}\label{eq1}
[N^{\Lam}(w,\nu,t)]_{\nu_t}q_{\nu_t}^{N^{\Lam}(w,\nu,t)-1}=\sum_{a=0}^{N^\Lam(w,\nu,t)-1}q_{\nu_t}^{2a} ;
\end{equation}
If $N^{\Lam}(w,\nu,t)<0$, then \begin{equation}\label{eq2}
[N^{\Lam}(w,\nu,t)]_{\nu_t}q_{\nu_t}^{N^{\Lam}(w,\nu,t)-1}=-\sum_{a=1}^{-N^\Lam(w,\nu,t)}q_{\nu_t}^{-2a} .
\end{equation}

Those integers $N^{\Lam}(w,\nu,t)$ could be negative or zero. Note that we always have $\sum\limits_{w\in\Sym(\nu,\nu')}\prod\limits_{t=1}^{n}N^\Lam(w,\nu,t)\geq 0$ as it is the dimension of a subspace by Corollary \ref{maincor1}. However, from the formula $\sum\limits_{w\in\Sym(\nu,\nu')}\prod\limits_{t=1}^{n}N^\Lam(w,\nu,t)$ itself, it is surprising to us why it is always non-negative.

The identity (\ref{eq1}) indicates that one might be able to obtain a monomial basis of $\RR^\Lam(\beta)$ of the form $\{e(\nu')\psi_wy_1^{c_1}\cdots y_n^{c_n}e(\nu)|0\leq c_t<N^\Lam(w,\nu,t),\forall\,1\leq t\leq n\}$.
The following example shows that this is not the case.

\begin{examp} Let $\mathscr{H}^{0}_{\ell,n}$ be the cyclotomic nilHecke algebra with level $\ell$ and size $n$. That is, $\mathscr{H}^{0}_{\ell,n}=\RR^\Lam(\beta)$ with $\Lam=\ell\Lam_0$, $\beta=n\alpha_0$. We consider the special case when $\ell=5, n=2$. Then $\Lam=5\Lam_0, \nu=(0,0)$ and $\Sym(\nu,\nu)=\{1,s_{1}\}$. By direct calculation, one gets that
$$
N^\Lam(1,\nu,1)=5,\,\, N^\Lam(1,\nu,2)=3,\,\, N^\Lam(s_1,\nu,1)=5,\,\,N^\Lam(s_1,\nu,2)=5.  $$
On the other hand, by \cite[Proposition 7]{HL} and \cite[Lemma 2.20]{HuL}, we have $$
\sum_{\substack{k_{1}+k_{2}=5-2+1=4}}x_{1}^{k_{1}}x_{2}^{k_{2}}=0. $$
Thus the elements in the set $\{\psi_{s_1}x_{1}^{a_{1}}x_{2}^{a_{2}}e(\nu)|0\leq a_{t}<N^\Lam(s_1,\nu,t)=5,\,t=1,2\}$ are $K$-linearly dependent.
\end{examp}

\smallskip
\subsection{A second formula for the dimension of $e(\nu)\RR^\Lam(\beta)e(\nu)$}

Let $\beta\in Q_n^+$ and $\nu\in I^{\beta}$. We can always write\begin{equation}\label{nu2}
\nu=(\nu_1,\cdots,\nu_n)=(\underbrace{\nu^{1},\nu^{1},\cdots,\nu^{1}}_{b_{1}\,copies},\cdots,\underbrace{\nu^{p},\nu^{p},\cdots,\nu^{p}}_{b_{p}\,copies}),
\end{equation}
where $p\in\N$, $b_1,\cdots,b_p\in\N$ with $\sum_{i=1}^{p}b_i=n$ and $\nu^{j}\neq\nu^{j+1}$ for any $1\leq j<p$.
The purpose of this subsection is to give a second formula for the dimension of $e(\nu)\RR^\Lam(\beta)e(\nu)$.

Define the set $$
\Sigma_n:=\bigl\{(k_1,\cdots,k_n)\in\Z^n\bigm|k_j\in\{0,1,\cdots,j-1\},\forall\,1\leq j\leq n\bigr\}.
$$
Consider the map $$\begin{aligned}
\theta_n: \Sym_n&\rightarrow \Sigma_n,\\
w&\mapsto \bigl(|J_w^{<1}|,\cdots,|J_w^{<n}|\bigr).
\end{aligned}
$$
It is clear that $\theta_n$ is well-defined by the definition of $J_w^{<t}$.

\begin{lem}\label{bij} With the above definitions and notations, we have that the map $\theta_n$ is a bijection.
\end{lem}

\begin{proof} Since both $\Sym_n$ and $\Sigma_n$ have cardinality $n!$, to prove the lemma, it suffices to show that $\theta_n$ is injective.

Let $w,u\in\Sym_n$ with $\theta_n(w)=\theta_n(u)$. Suppose that $u\neq w$. Let $1\leq t\leq n$ be the unique integer such that $w(t)\neq u(t)$ and $w(i)=u(i)$ for any $t<i\leq n$. Assume that $w(t)<u(t)$. Then $w(t)=u(m_t)$ for some $m_t\in\{1,2,\cdots,t-1\}$. Note that if $1\leq j<t$ and $w(j)<w(t)$, then for these $j$ we have $u(m_j)=w(j)<w(t)<u(t)$ for some $1\leq m_j<t$. It follows that $|J_w^{<t}|\leq |J_u^{<t}|-1$, a contradiction. In a similar (and symmetric) argument one can show that $u(t)<w(t)$ can not happen. Thus we get that $w(t)=u(t)$ which is a contradiction. This proves that $\theta_n$ is injective. Hence we complete the proof of the lemma.
\end{proof}

Let $\nu\in I^\beta$ be given as in (\ref{nu2}). For $0\leq t\leq p$, we define $$
b_0:=0,\quad c_t:=\sum_{i=0}^{t}b_i,\quad \Sym_{\fb}:=\Sym_{\{1,\cdots,c_1\}}\times\Sym_{\{c_1+1,\cdots,c_2\}}\times\cdots\times\Sym_{\{c_{p-1}+1,\cdots,n\}}.
$$
Let $\mathcal{D}_{\fb}$ be the set of minimal length left $\Sym_{\fb}$-coset representatives in $\Sym_{n}$. Set $\mathcal{D}(\nu):=\mathcal{D}_{\fb}\cap\Sym(\nu,\nu)$. Then we have $\Sym(\nu,\nu)=\mathcal{D}(\nu)\Sym_{\fb}$.

\begin{lem}\label{independ} Let $k$ be an integer with $c_{i-1}<k\leq c_{i}$, where $1\leq i\leq p$. Let $d\in\mathcal{D}(\nu)$, $w=w_{1}\times\cdots\times
w_{p}$, where $w_{j}\in \Sym_{\{c_{j-1}+1,\cdots,c_j\}},\,\,\forall\,1\leq j\leq p$. Then we have that $$N^\Lam(dw,\nu,k)=N^\Lam(d,\nu,w_i(k))-2|\tilde{J}_{w_i}^{<k}|+2(w_i(k)-c_{i-1}-1), $$
where $$
 \tilde{J}_{w_i}^{<k}:=\{c_{i-1}+1\leq a<k|w_i(a)<w_i(k)\}. $$
In particular, $N^\Lam(dw,\nu,k)$ does not depend on $w_j$ for any $1\leq j\neq i\leq p$.
\end{lem}

\begin{proof}
By Definition \ref{keydfn1} and the definition of $\mathcal{D}(\nu)$, we have  $$
\begin{aligned}
J_{dw}^{<k}&=\cup_{j<i}\{c_{j-1}+1\leq a\leq c_j|dw(a)<dw(k)\}\cup\{c_{i-1}+1\leq a<k|dw(a)<dw(k)\}\\
&=\cup_{j<i}\{c_{j-1}+1\leq a\leq c_j|dw_j(a)<dw_i(k)\}\cup\{c_{i-1}+1\leq a<k|dw_i(a)<dw_i(k)\}\\
&=\cup_{j<i}\{c_{j-1}+1\leq a\leq c_j|dw_j(a)<dw_i(k)\}\cup\{c_{i-1}+1\leq a<k|w_i(a)<w_i(k)\}\\
&=\cup_{j<i}\{c_{j-1}+1\leq a\leq c_j|dw_j(a)<dw_i(k)\}\cup\tilde{J}_{w_i}^{<k}.
\end{aligned}$$
and $$
\begin{aligned}
J_{d}^{<w_i(k)}&=\cup_{j<i}\{c_{j-1}+1\leq a\leq c_j|d(a)<dw_{i}(k)\}\cup\{c_{i-1}+1\leq a<w_{i}(k)|d(a)<dw_{i}(k)\}\\
&=\cup_{j<i}\{c_{j-1}+1\leq a\leq c_j|d(a)<dw_{i}(k)\}\cup\{c_{i-1}+1\leq a<w_{i}(k)|a<w_{i}(k)\}.
\end{aligned}
$$
Since the map $$\begin{aligned}
\gamma_{j}: \{c_{j-1}+1\leq a\leq c_j|dw_j(a)<dw_i(k)\}&\rightarrow \{c_{j-1}+1\leq a\leq c_j|d(a)<dw_{i}(k)\},\\
 a&\mapsto w_j(a)
\end{aligned}
$$ is a well-defined bijection for $j<i$, we have  $$
|\{c_{j-1}+1\leq a\leq c_j|dw_j(a)<dw_i(k)\}|=|\{c_{j-1}+1\leq a\leq c_j|d(a)<dw_{i}(k)\}|
$$ when $j<i$.
Now the result follows directly from (\ref{Ndef}).
\end{proof}

\begin{rem}\label{rem2} The significance of the above lemma lies in that it means the integer $N^\Lam(dw,\nu,k)$ depends only on the interval $(c_{i-1},c_i]$ to which $k$ belongs and the element $w_i$, but not on the elements
$w_j$ for any $j\in\{1,2,\cdots,p\}\setminus\{i\}$.
\end{rem}

\begin{dfn} Let $\nu\in I^\beta$ be given as in (\ref{nu2}). For any $d\in\mathcal{D}(\nu)$, $1\leq i\leq p$ and $c_{i-1}<k\leq c_i$, we define \begin{equation}
\label{Ntilde}
\widetilde{N}^\Lam(d,\nu,k):=N^\Lam(d,\nu,k)+k-c_{i-1}-1 .
\end{equation}
\end{dfn}

\begin{thm}\label{mainthm1b} Let $\nu\in I^\beta$ be given as in (\ref{nu2}). Then $$
\dim\,e(\nu) \RR^\Lam({\beta})e(\nu)=\Bigl(\prod_{i=1}^{p}b_{i}!\Bigr)\sum_{d\in \mathcal{D}(\nu)}\Bigl(\prod_{t=1}^{n}\widetilde{N}^\Lam(d,\nu,t)\Bigr). $$
\end{thm}

\begin{proof}
By Corollary \ref{maincor1} and Lemma \ref{independ}, we have $$
\begin{aligned}
&\quad\,\dim e(\nu) \RR^\Lam({\beta})e(\nu)\\
=&\sum\limits_{w\in\Sym(\nu,\nu)}\prod\limits_{t=1}^{n}N^\Lam(w,\nu,t)=\sum\limits_{d\in\mathcal{D}(\nu)}\sum\limits_{w\in d\Sym_\fb}\prod\limits_{t=1}^{n}N^\Lam(w,\nu,t)\\
&=\sum\limits_{d\in\mathcal{D}(\nu)}\sum_{w\in d\Sym_\fb}\prod_{i=1}^{p}\prod\limits_{t=c_{i-1}+1}^{c_i}N^\Lam(w,\nu,t)\\
&=\sum\limits_{d\in\mathcal{D}(\nu)}\sum_{\substack{w_j\in \Sym_{\{c_{j-1}+1,\cdots,c_{j}\}}\\ \forall 1\leq j\leq p}}\prod_{i=1}^{p}\prod\limits_{t=c_{i-1}+1}^{c_i}N^\Lam(dw_1\cdots w_p,\nu,t)\\
&=\sum\limits_{d\in\mathcal{D}(\nu)}\sum_{\substack{w_j\in \Sym_{\{c_{j-1}+1,\cdots,c_{j}\}}\\ \forall 1\leq j\leq p}}\prod_{i=1}^{p}\prod\limits_{t=c_{i-1}+1}^{c_i}N^\Lam(dw_i,\nu,t)\\
&=\sum\limits_{d\in\mathcal{D}(\nu)}\prod_{i=1}^{p}\sum_{w_i\in \Sym_{\{c_{i-1}+1,\cdots,c_{i}\}}}\prod\limits_{t=c_{i-1}+1}^{c_i}(N^\Lam(d,\nu,w_i(t))-2|\tilde{J}_{w_i}^{<t}|+2(w_i(t)-c_{i-1}-1))\\
\end{aligned}
$$
Note that the map $$\tilde{\gamma}_{i}: \tilde{J}_{w_i}^{<w_i^{-1}(k)}\rightarrow \tilde{J}_{w_i^{-1}}^{<k},\,\, a\mapsto w_i(a),$$ is a well-defined bijection for $c_{i-1}+1\leq k\leq c_i$. In particular, $|\tilde{J}_{w_i}^{<w_i^{-1}(k)}|=|\tilde{J}_{w_i^{-1}}^{<k}| $ for $c_{i-1}+1\leq k\leq c_i$. Combing this equality with the bijection in Lemma \ref{bij}, we get that $$
\begin{aligned}
&\sum_{w_i\in \Sym_{\{c_{i-1}+1,\cdots,c_{i}\}}}\prod\limits_{t=c_{i-1}+1}^{c_i}(N^\Lam(d,\nu,w_i(t))-2|\tilde{J}_{w_i}^{<t}|+2(w_i(t)-c_{i-1}-1))\\
&=\sum_{w_i\in \Sym_{\{c_{i-1}+1,\cdots,c_{i}\}}}\prod\limits_{k=c_{i-1}+1}^{c_i}(N^\Lam(d,\nu,k)-2|\tilde{J}_{w_i}^{<w_i^{-1}(k)}|+2(k-c_{i-1}-1))\\
&=\sum_{w_i\in \Sym_{\{c_{i-1}+1,\cdots,c_{i}\}}}\prod\limits_{k=c_{i-1}+1}^{c_i}(N^\Lam(d,\nu,k)-2|\tilde{J}_{w_i^{-1}}^{<k}|+2(k-c_{i-1}-1))\\
&=\sum_{w_i\in \Sym_{\{c_{i-1}+1,\cdots,c_{i}\}}}\prod\limits_{k=c_{i-1}+1}^{c_i}(N^\Lam(d,\nu,k)-2|\tilde{J}_{w_i}^{<k}|+2(k-c_{i-1}-1))\\
&=\prod\limits_{k=c_{i-1}+1}^{c_i}\Bigl(N^\Lam(d,\nu,k)+2(k-c_{i-1}-1)+N^\Lam(d,\nu,k)-2+2(k-c_{i-1}-1)\\
&\qquad +\cdots+N^\Lam(d,\nu,k)-2(k-c_{i-1}-1)+2(k-c_{i-1}-1)\Bigr)\\
&=\prod\limits_{k=c_{i-1}+1}^{c_i}(k-c_{i-1})(N^\Lam(d,\nu,k)+k-c_{i-1}-1)\\
&=\prod\limits_{k=c_{i-1}+1}^{c_i}(k-c_{i-1})\tilde{N}^\Lam(d,\nu,k)\\
&=b_i!\prod\limits_{k=c_{i-1}+1}^{c_i}\tilde{N}^\Lam(d,\nu,k).
\end{aligned}
$$
Combining this equality with the equality obtained in the first paragraph of this proof, we prove the theorem.
\end{proof}

\begin{lem}\label{Lemma 4} Let $t\in\Z^{\geq 1}$ and $l\in\Z$. Then $$
\sum_{k=0}^{t-1}[l-2k]q^{l-t}=[t](1+q^{2}+\cdots+q^{2(l-t)}). $$
\end{lem}

\begin{proof} It suffices to show that $$
\sum_{k=0}^{t-1}(q^{l-2k}-q^{-(l-2k)})q^{l-t}=(q^{t}-q^{-t})(1+q^{2}+\cdots+q^{2(l-t)}).
$$
In fact, the left-hand side of the above equality is equal to $$\begin{aligned}
\sum_{k=0}^{t-1}q^{2l-t}q^{-2k}-\sum_{k=0}^{t-1}q^{-t}q^{2k}
&=q^{2l-t}\frac{1-q^{-2t}}{1-q^{-2}}-q^{-t}\frac{1-q^{2t}}{1-q^{2}}\\
&=\frac{q^{2l-3t+2}-q^{2l-t+2}}{1-q^{2}}-\frac{q^{-t}-q^{t}}{1-q^{2}},\end{aligned} $$
while the right-hand side of the above equality is equal to $$
(q^{t}-q^{-t})\frac{1-q^{2(l-t+1)}}{1-q^{2}}.
$$
Hence, they are equal to each other.
\end{proof}

In the rest of this subsection we consider the cyclotomic nilHecke algebra $\mathscr{H}^{0}_{\ell,n}=\RR^\Lam(\beta)$ with $\Lam=\ell\Lam_0$ and $\beta=n\alpha_0$. In this case, by definition, we have $$
N^\Lam(w,\nu,t)=\ell-2|J_w^{<t}|,\quad N^\Lam(1,\nu,t)=\ell-2(t-1),\,\,\forall\,1\leq t\leq n .
$$
The bijection $\theta$ between $\Sym_n$ and $\Sigma_n$ established in Lemma \ref{bij} implies that \begin{equation}\label{swap}
\sum_{w\in\Sym_n}\prod_{t=1}^{n}[N^\Lam(w,\nu,t)]_{\nu_t}=\sum_{w\in\Sym_n}\prod_{t=1}^{n}[\ell-2|J_w^{<t}|]=\prod_{t=1}^{n}\sum_{k=0}^{t-1}[\ell-2k].
\end{equation}

Combining the above results with Theorem $\ref{mainthmA}$, we derive the following graded dimension formula for the cyclotomic nilHecke algebra $\mathscr{H}^{(0)}_{\ell,n}$.

\begin{cor}\label{maincor2a} Let $\Lam:=\ell\Lam_0, \beta=n\alpha_0$. We have
$$
\dim_{q}\,\mathscr{H}^{(0)}_{\ell,n}=\Bigl(\prod_{k=1}^{n}\frac{q^{-2k}-1}{q^{-2}-1}\Bigr)\Bigl(\prod_{t=1}^{n}(1+q^{2}+\cdots+q^{2(\ell-t)})\Bigr). $$
\end{cor}

\begin{proof} Applying Theorem \ref{mainthmA} in our special case $\Lam:=\ell\Lam_0, \beta=n\alpha_0$, we can get that $$\begin{aligned}
&\quad\,\dim_{q}\,\mathscr{H}^{(0)}_{\ell,n}=\sum_{w\in\Sym_n}\prod_{t=1}^{n}\bigl([\ell-2|J_w^{<t}|]q^{\ell-2t+1}\bigr)\\
&=q^{n(\ell-n)}\sum_{w\in\Sym_n}\prod_{t=1}^{n}[\ell-2|J_w^{<t}|]\\
&=q^{n(\ell-n)}\prod_{t=1}^{n}\sum_{k=0}^{t-1}[\ell-2k]\qquad\qquad\qquad\text{(by (\ref{swap}))}\\
&=q^{-n(\ell-n)/2}\prod_{t=1}^{n}\sum_{k=0}^{t-1}\Bigl([\ell-2k]q^{\ell-t}\Bigr)\\
&=q^{-n(n-1)/2}\prod_{t=1}^{n}\frac{(q^{t}-q^{-t})(1+q^{2}+\cdots+q^{2(\ell-t)})}{q-q^{-1}}\qquad\qquad        \text{(by Lemma \ref{Lemma 4})}\\
&=\Bigl(\prod_{k=1}^{n}\frac{q^{-2k}-1}{q^{-2}-1}\Bigr)\Bigl(\prod_{t=1}^{n}(1+q^{2}+\cdots+q^{2(\ell-t)})\Bigr).
\end{aligned}$$
This completes the proof of the corollary.
\end{proof}

Note that the above graded dimension formula for $\mathscr{H}^{(0)}_{\ell,n}$ also follows from \cite[Theorem 2.34]{HuL}. The polynomial $\prod_{k=1}^{n}\frac{q^{k}-1}{q-1}=\sum_{w\in\Sym_n}q^{\ell(w)}$ is the Poincare polynomial for the Iwahori-Hecke algebra $\HH_q(\Sym_n)$ associated to the symmetric group $\Sym_n$. Specializing $q$ to $1$, we obtain the following well-known dimension formula for the (ungraded) cyclotomic nilHecke algebra $\mathscr{H}^{(0)}_{\ell,n}$.

\begin{cor}\label{maincor2b}
$dim\,\mathscr{H}^{(0)}_{\ell,n}=n!\prod\limits_{j=0}^{n-1}(\ell-j).$
\end{cor}

\subsection{Criteria for $e(\nu)\neq 0$ in $\RR^\Lam(\beta)$}

In this subsection, we shall give some criteria for $e(\nu)\neq 0$ in $\RR^\Lam(\beta)$. In particular, we shall give a proof of Theorems \ref{mainthmB} here.

In the special cases of types $A_\ell^{(1)}$ and $A_\infty$, it was shown in \cite[Lemma 4.1]{HM} that $e(\nu)\neq 0$ in $\RR^\Lam(\beta)$ if and only if $\nu=(\nu_1,\cdots,\nu_n)$ is the residue sequence of a standard tableau in the subset $\mathscr{P}_\beta^\Lam$ of multi-partitions of $n$ determined by $\beta$. Similar criteria in the cases of types $C_\ell^{(1)}$ and $C_\infty$ can be obtained from \cite[Theorem 2.5]{APS}. These are not effective criteria in the sense that one has to check many standard tableaux in $\mathscr{P}_\beta^\Lam$. Our second main result Theorem \ref{mainthmB} of this paper solves the problems on determining when the KLR idempotent $e(\nu)\neq 0$ in $\RR^\Lam(\beta)$ for {\it arbitrary} symmetrizable Cartan matrix.

\medskip
\noindent
{\textbf{Proof of Theorem \ref{mainthmB}}}:  Let $\Lam\in P^+$, $\beta\in Q^+$ and $\nu=(\nu_1,\cdots,\nu_n)\in I^\beta$. It is clear that $e(\nu)\neq 0$ in $\RR^\Lam(\beta)$ if and only if $e(\nu)\RR^{\Lambda}(\beta)e(\nu)\neq 0$. Thus Theorem \ref{mainthmB} follows from Corollary \ref{maincor1}.
\qed
\medskip

Using our second version of the dimension formula for $e(\nu)\RR^{\Lambda}(\beta)e(\nu)$ given in Theorem \ref{mainthm1b}, we also obtain in Theorem \ref{mainthmB2} a second simplified (or divided power) version of the criterion for the KLR idempotent $e(\nu)$ to be nonzero in $\RR^\Lam(\beta)$.
As in the beginning of last subsection, we can always write\begin{equation}
\nu=(\nu_1,\cdots,\nu_n)=(\underbrace{\nu^{1},\nu^{1},\cdots,\nu^{1}}_{b_{1}\,copies},\cdots,\underbrace{\nu^{p},\nu^{p},\cdots,\nu^{p}}_{b_{p}\,copies}),
\end{equation}
where $p\in\N$, $b_1,\cdots,b_p\in\N$ with $\sum_{i=1}^{p}b_i=n$ and $\nu^{j}\neq\nu^{j+1}$ for any $1\leq j<p$. Let $\widetilde{N}^\Lam(d,\nu,t)$ be the integer as defined in (\ref{Ntilde}) and $\mathcal{D}(\nu)$ be defined as before.

\begin{thm}\label{mainthmB2} Let $\Lam\in P^+$, $\beta\in Q^+$ and $\nu=(\nu_1,\cdots,\nu_n)\in I^\beta$. Then $e(\nu)\neq 0$ in $\RR^\Lam(\beta)$ if and only if $$
\sum_{d\in \mathcal{D}(\nu)}\prod_{t=1}^{n}\widetilde{N}^\Lam(d,\nu,t)\neq 0 .
$$
\end{thm}

\begin{proof} The proof is the same as the proof of Theorem \ref{mainthmB} by using Theorem \ref{mainthm1b}.
\end{proof}

To sum all, we have given two criteria for $e(\nu)\neq 0$ in $\RR^\Lam(\beta)$ in this subsection. A third criterion (Corollary \ref{maincorC3}) for $e(\nu)\neq 0$ in $\RR^\Lam(\beta)$ will be given at the end of the next section.

\bigskip
\section{Level reduction for dimension formulae}

In this section we shall give a second application---level reduction for our dimension formula, which reveals some surprising connections between the dimension of the higher level cyclotomic quiver Hecke algebras with a sum of some products of the dimensions of some lower level cyclotomic quiver Hecke algebras. In particular, we shall give the proof of the fourth main result Theorem \ref{mainthmC} of this paper.

For any $\nu=(\nu_1,\cdots,\nu_n)\in I^n$, we define \begin{equation}\label{betanu}
\beta_{\nu}:=\sum_{i=1}^{n}\alpha_{\nu_i},\quad |\nu|:=n. \end{equation}
Let $\mathcal{D}_{(k,n-k)}$ be the set of minimal length left coset representatives of $\Sym_{(k,n-k)}$ in $\Sym_n$. We define $D^2(n)$ to be the set of all $(k,n-k)$-shuffles of $(1,2,\cdots,n)$ for $k=0,1,\cdots,n$. That is,  $$ D^2(n)=\Bigl\{\bigl((w(1),\cdots,w(k)),(w(k+1),\cdots,w(n))\bigr)\Bigm|
\begin{matrix}\text{$w\in\mathcal{D}_{(k,n-k)}$,}\\ \text{$k=0,1,\cdots,n$}\end{matrix}\Bigr\}.
$$
In particular, we always have $|D^2(n)|=2^n$.

\begin{dfn}\label{ssize} Let $\nu=(\nu_1,\cdots,\nu_n)\in I^n$. For any $k$-tuple ${\bf s}=(s_1,s_2,\cdots,s_k)$ of integers with $1\leq s_1<\cdots<s_k\leq n$, we define $$
|{\bf s}|:=k,\quad \nu_{\bf s}:=(\nu_{s_1},\cdots,\nu_{s_k}). $$
For any $\mu\in I^n$, we define $$
D^2(\nu,\mu):=\bigl\{\big(({\bf s}^1,{\bf s}^2),({\bf t}^1,{\bf t}^2)\big)\in D^2(n)\times D^2(n)\bigm|\beta_{\nu_{{\bf s}^i}}=\beta_{\mu_{{\bf t}^i}},\,\,i=1,2\bigr\}.
$$
\end{dfn}

Let $\big(({\bf s}^1,{\bf s}^2),({\bf t}^1,{\bf t}^2)\big)\in D^2(\nu,\mu)$. By construction, each $w_1\times w_2\in \Sym(\nu_{{\bf s}^1},
\mu_{{\bf t}^1})\times \Sym(\nu_{{\bf s}^2},\mu_{{\bf t}^2})$
can determine a unique element $w\in \Sym(\nu,\mu)$. Hence, we can get a canonical map:$$\tau:\,\bigsqcup_{\substack{\big(({\bf s}^1,{\bf s}^2),({\bf t}^1,{\bf t}^2)\big)\in D^2(\nu,\mu)}}\Bigl(\Sym(\nu_{{\bf s}^1},\mu_{{\bf t}^1})\times \Sym(\nu_{{\bf s}^2},\mu_{{\bf t}^2})\Bigr)\rightarrow\,\Sym(\nu,\mu).
$$

We can visualize any $w\in\Sym(\nu,\mu)$ as a planar diagram as follows: the diagram has two rows of vertices, each of them are labelled by $1,2,\cdots,n$, and there is an edge connecting the vertex $i$ in the top row with the vertex $j$ in the bottom row if and only if $w(i)=j$ and $\nu_{i}=\mu_j$.
For ${\bf s}^1=(s_1,\cdots,s_k), {\bf t}^1=(t_1,\cdots,t_k)$ with $1\leq s_1<s_2<\cdots<s_k\leq n, 1\leq t_1<t_2<\cdots<t_k\leq n$,
any $w_1\in\Sym(\nu_{{\bf s}^1},\mu_{{\bf t}^1})$ can be identified as a planar diagram as follows: the diagram has two rows of vertices, the top row vertices are labelled by $s_1,s_2,\cdots,s_k$, the bottom row vertices are labelled by $t_1,t_2,\cdots,t_k$, and there is an edge connecting the vertex $s_i$ in the top row with the vertex $t_j$ in the bottom row if and only if $w_1(i)=j$ and $\nu_{s_{i}}=\mu_{t_j}$. Similarly, we have the planar diagram for $({\bf s}^2,{\bf t}^2)$. Then the map $\tau$ is the native way to incorporate the two planar diagrams associated to $({\bf s}^1,{\bf t}^1),({\bf s}^2,{\bf t}^2)$ to a new diagram without breaking any edges in the diagram.

\begin{lem}\label{surjective}\begin{enumerate}
\item Let $\mu,\nu\in I^n$ and $w\in\Sym(\nu,\mu)$. Then for each ${\bf s}:=({\bf s}^1,{\bf s}^2)\in D^2(n)$, there exists a unique $w_1\in\Sym_{|{\bf s}^1|}$, a unique $w_2\in\Sym_{|{\bf s}^2|}$ and a unique
$({\bf t}^1,{\bf t}^2)\in D^2(n)$, such that $w_1\times w_2\in \Sym(\nu_{{\bf s}^1},\nu_{{\bf t}^1})\times \Sym(\nu_{{\bf s}^2},\nu_{{\bf t}^2})$ and $\tau(w_1\times w_2)=w$. In particular, $\tau$ is surjective;
\item For each $w\in\Sym(\nu,\mu)$, the cardinality of $\tau^{-1}(w)$ is $2^n$.
\end{enumerate}
\end{lem}

\begin{proof} Let $w\in\Sym(\nu,\mu)$ and ${\bf s}:=({\bf s}^1,{\bf s}^2)\in D^2(n)$, where ${\bf s}^1=(i_1,\cdots,i_a)$,
${\bf s}^2=(\hat{i}_1,\cdots,\hat{i}_{n-a})$, $1\leq i_1<\cdots<i_a\leq n$, $1\leq\hat{i}_1<\cdots<\hat{i}_{n-a}\leq n$. Then ${\bf t}^1=(j_1,\cdots,j_a)$ is the unique rearrangement of $(w(i_1),\cdots,w(i_a))$ such that $1\leq j_1<\cdots<j_a\leq n$, while ${\bf t}^2=(\hat{j}_1,\cdots,\hat{j}_{n-a})$ is the unique rearrangement of $(w(\hat{i}_1),\cdots,w(\hat{i}_{n-a}))$ such that $1\leq \hat{j}_1<\cdots<\hat{j}_{n-a}\leq n$. We set $w_1\in\Sym_{a}$ to be the unique element such that $j_t=w(i_{w_1^{-1}(t)})$ for any $1\leq t\leq a$, while $w_2\in\Sym_{n-a}$ is the unique element such that
$\hat{j}_t=w(\hat{i}_{w_2^{-1}(t)})$ for any $1\leq t\leq n-a$. This proves 1). Now 2) follows from 1) and the fact that $|D^2(n)|=2^n$.
\end{proof}

\begin{dfn} Let $\mu,\nu\in I^n$ and $w\in\Sym(\nu,\mu)$. For $i\in\{1,2\}$, we define $w_{{\bf s},i}\in\Sym_{|{\bf s}^i|}$ to be the unique element $w_i$ determined by $w$ and ${\bf s}=({\bf s}^1,{\bf s}^2)$ which was introduced in Lemma \ref{surjective}.
\end{dfn}

\begin{thm}\label{mainthmC1}
Let $\mu,\nu\in I^n$. Suppose $\Lam=\Lam^1+\Lam^2$, where $\Lam^1,\Lam^2\in P^+$. Then
$$\begin{aligned}\dim e(\nu)\RR^\Lam(\beta) e(\mu)&=\sum_{\substack{\big(({\bf s}^1,{\bf s}^2),({\bf t}^1,{\bf t}^2)\big)\in D^2(\nu,\mu)}} \dim e(\nu_{{\bf s}^1})\RR^{\Lam^1}(\beta_{\nu_{{\bf s}^1}})e(\mu_{{\bf t}^1})\\
&\qquad\qquad\qquad \times\dim e(\nu_{{\bf s}^2})\RR^{\Lam^2}(\beta_{\nu_{{\bf s}^2}})e(\mu_{{\bf t}^2}). \end{aligned}$$
\end{thm}

\begin{proof}
By dimension formula in Corollary \ref{maincor1} and Lemma \ref{surjective}, we have:$$
\begin{aligned}{\text{RHS}}=&\sum_{\substack{\big(({\bf s}^1,{\bf s}^2),({\bf t}^1,{\bf t}^2)\big)\in D^2(\nu,\mu)}} \sum_{\substack{w_1\in \Sym(\nu_{{\bf s}^1},\mu_{{\bf t}^1})\\w_2\in\Sym(\nu_{{\bf s}^2},\mu_{{\bf t}^2})}}\prod_{\substack{k_1=1,\cdots,|{\bf s}^1|,\\k_2=1,\cdots,|{\bf s}^2|}}N^{\Lam^1}(w_1,\nu_{{\bf s}^1},k_1)N^{\Lam^2}(w_2,\nu_{{\bf s}^2},k_2)\\
=&\sum_{w\in \Sym(\nu,\mu)}\sum_{({\bf s}^1,{\bf s}^2)\in D^2(n)}\prod_{\substack{k_1=1,\cdots,|{\bf s}^1|,\\k_2=1,\cdots,|{\bf s}^2|}}N^{\Lam^1}(w_{{\bf s},1},\nu_{{\bf s}^1},k_1)N^{\Lam^2}(w_{{\bf s},2},\nu_{{\bf s}^2},k_2).
\end{aligned}$$
To prove the theorem, it suffices to show for each $w\in \Sym(\nu,\mu)$, \begin{equation}\label{claimrd11}\sum_{({\bf s}^1,{\bf s}^2)\in D^2(n)}\prod_{\substack{k_1=1,\cdots,|{\bf s}^1|,\\k_2=1,\cdots,|{\bf s}^2|}}N^{\Lam^1}(w_{{\bf s},1},\nu_{{\bf s}^1},k_1)N^{\Lam^2}(w_{{\bf s},2},\nu_{{\bf s}^2},k_2)=\prod_{t=1}^{n}N^{\Lam}(w,\nu,t).
\end{equation}

To see this, we consider the following map:  $$f_n: D^2(n)\rightarrow D^2(n)$$
\[({\bf s}^1,{\bf s}^2)\mapsto
\begin{cases}
({\bf s}^1\setminus\{n\},{\bf s}^2\cup\{n\}),& \text{if $n\in {\bf s}^1$;}\\
({\bf s}^1\cup\{n\},{\bf s}^2\setminus\{n\}),& \text{if $n\in {\bf s}^2$,}
\end{cases}\]
where ${\bf s}^i\setminus\{n\}$ means that we remove the integer $n$ from ${\bf s}^i$ and ${\bf s}^i\cup\{n\}$ means we add the integer $n$ to the end of ${\bf s}^i$. It's easy to see $f_n$ is a well-defined involution. For any $({\bf s}^1,{\bf s}^2)\in D^2(n)$, we set $({\bf \widetilde{s}}^1,{\bf \widetilde{s}}^2):=f_n({\bf s}^1,{\bf s}^2)$. Note that if $n\in {\bf s}^i$ then $n$ must sit at the end of ${\bf s}^i$. Clearly, by the discussion in the paragraph above Lemma \ref{surjective} and Definition \ref{keydfn1}, $$\begin{aligned}
&\prod_{\substack{k_1=1,\cdots,|{\bf s}^1\setminus\{n\}|,\\k_2=1,\cdots,|{\bf s}^2\setminus\{n\}|}}N^{\Lam^1}(w_{{\bf s},1},\nu_{{\bf s}^1},k_1)N^{\Lam^2}(w_{{\bf s},2},\nu_{{\bf s}^2},k_2)\\
&\qquad =\prod_{\substack{k_1=1,\cdots,|{\bf \widetilde{s}}^1\setminus\{n\}|,\\k_2=1,\cdots,|{\bf \widetilde{s}}^2\setminus\{n\}|}}N^{\Lam^1}(w_{{\bf\widetilde{s}},1},\nu_{{\bf \widetilde{s}}^1},k_1)N^{\Lam^2}(w_{{\bf\widetilde{s}},2},\nu_{{\bf \widetilde{s}}^2},k_2) .
\end{aligned}
$$

If $n\in {\bf s}^1$, then $$
\begin{aligned}
&\prod_{\substack{k_1=1,\cdots,|{\bf s}^1|,\\k_2=1,\cdots,|{\bf s}^2|}}N^{\Lam^1}(w_{{\bf s},1},\nu_{{\bf s}^1},k_1)N^{\Lam^2}(w_{{\bf s},2},\nu_{{\bf s}^2},k_2)\\
&\qquad\qquad+\prod_{\substack{k_1=1,\cdots,|{\bf \widetilde{s}}^1|,\\k_2=1,\cdots,|{\bf \widetilde{s}}^2|}}N^{\Lam^1}(w_{{\bf\widetilde{s}},1},\nu_{{\bf \widetilde{s}}^1},k_1)N^{\Lam^2}(w_{{\bf\widetilde{s}},2},\nu_{{\bf \widetilde{s}}^2},k_2)\\
=&N^{\Lam^1}(w_{{\bf s},1},\nu_{{\bf s}^1},|{\bf s}^1|)\prod_{\substack{k_1=1,\cdots,|{\bf s}^1\setminus\{n\}|,\\k_2=1,\cdots,|{\bf s}^2\setminus\{n\}|}}N^{\Lam^1}(w_{{\bf s},1},\nu_{{\bf s}^1},k_1)N^{\Lam^2}(w_{{\bf s},2},\nu_{{\bf s}^2},k_2)\\
&\qquad\qquad +N^{\Lam^2}(w_{{\bf\widetilde{s}},2},\nu_{{\bf \widetilde{s}}^2},|{\bf \widetilde{s}}^2|)\prod_{\substack{k_1=1,\cdots,|{\bf \widetilde{s}}^1\setminus\{n\}|,\\k_2=1,\cdots,|{\bf \widetilde{s}}^2\setminus\{n\}|}}N^{\Lam^1}(w_{{\bf\widetilde{s}},1},\nu_{{\bf \widetilde{s}}^1},k_1)N^{\Lam^2}(w_{{\bf\widetilde{s}},2},\nu_{{\bf \widetilde{s}}^2},k_2).
\end{aligned}$$
By assumption, $\tau(w_{{\bf s},1}\times w_{{\bf s},2})=w=\tau(w_{{\bf\widetilde{s}},1}\times w_{{\bf\widetilde{s}},2})$ and $n\in {\bf s}^1\cap {\bf \widetilde{s}}^2$. To simplify the notations, we set $$\begin{aligned}
& a:=|\nu_{{\bf s}^1}|,\,\,\, {\bf s}^1=(i_1,\cdots,i_{a-1},n),\,\,\, {\bf s}^2=(\hat{i}_1,\cdots,\hat{i}_{n-a}),\,\,\mu=w\nu=(\mu_1,\cdots,\mu_n),\\
& w_{{\bf s},1}\nu_{{\bf s}^1}=(\mu_{j_1},\cdots,\mu_{j_a}),\,\, w_{{\bf s},2}\nu_{{\bf s}^2}=(\mu_{\hat{j}_1},\cdots,\mu_{\hat{j}_{n-a}}),
\end{aligned}
$$
where $((j_1,\cdots,j_a),(\hat{j}_1,\cdots,\hat{j}_{n-a}))$ is an $(a,n-a)$-shuffle of $(1,2,\cdots,n)$.
Then $$\begin{aligned}
&\widetilde{\bf s}^1=(i_1,\cdots,i_{a-1}),\,\,\, \widetilde{\bf s}^2=(\hat{i}_1,\cdots,\hat{i}_{n-a},n),\\
& w_{\widetilde{\bf s},1}\nu_{\widetilde{\bf s}^1}=(\mu_{j_1},\cdots,\mu_{j_{a-1}}),\,\, w_{\widetilde{\bf s},2}\nu_{\widetilde{\bf s}^2}=(\mu_{\hat{j}_1},\cdots,\mu_{\hat{j}_{k}},\mu_{j_a},\mu_{\hat{j}_{k+1}},\cdots,\mu_{\hat{j}_{n-a}}), \end{aligned}
$$
where $1\leq k\leq n-a$ is such that $\hat{j}_1<\cdots<\hat{j}_k<j_a<\hat{j}_{k+1}<\cdots<\hat{j}_{n-a}$.

Given $1\leq k\leq n$ with $w(k)<w(n)$, we have either $k=i_t$ for some $1\leq t<a$, or $k=\hat{i}_l$ for some $1\leq l\leq n-a$. In the former case, $w(i_t)=j_{w_{{\bf s},1}(t)}$, $w(n)=j_{w_{{\bf s},1}(a)}$, and thus $w(i_t)<w(n)$ implies that
$w_{{\bf s},1}(t)<w_{{\bf s},1}(a)$; in the latter case, $w(\hat{i}_l)=\hat{j}_{w_{\widetilde{\bf s},2}(l)}$, and thus $w(\hat{i}_l)<w(n)$ implies that
$w_{\widetilde{\bf s},2}(l)<j_a=w_{\widetilde{\bf s},2}(n-a+1)$. As a result, we see from Definition \ref{keydfn1} that $$N^{\Lam^1}(w_{{\bf s},1},\nu_{{\bf s}^1},|{\bf s}^1|)+N^{\Lam^2}(w_{{\bf\widetilde{s}},2},\nu_{{\bf \widetilde{s}}^2},|{\bf \widetilde{s}}^2|)=N^{\Lam}(w,\nu,n).$$
We get that $$\begin{aligned}
&\quad \prod_{\substack{k_1=1,\cdots,|{\bf s}^1|,\\k_2=1,\cdots,|{\bf s}^2|}}N^{\Lam^1}(w_{{\bf s},1},\nu_{{\bf s}^1},k_1)N^{\Lam^2}(w_{{\bf s},2},\nu_{{\bf s}^2},k_2)\\
&\qquad\qquad +\prod_{\substack{k_1=1,\cdots,|{\bf \widetilde{s}}^1|,\\k_2=1,\cdots,|{\bf \widetilde{s}}^2|}}N^{\Lam^1}(w_{\widetilde{\bf s},1},\nu_{{\bf \widetilde{s}}^1},k_1)N^{\Lam^2}(w_{\widetilde{\bf s},2},\nu_{{\bf \widetilde{s}}^2},k_2)\\
&=N^{\Lam}(w,\nu,n)\prod_{\substack{k_1=1,\cdots,|{\bf s}^1\setminus\{n\}|,\\k_2=1,\cdots,|{\bf s}^2\setminus\{n\}|}}N^{\Lam^1}(w_{{\bf s},1},\nu_{{\bf s}^1},k_1)N^{\Lam^2}(w_{{\bf s},2},\nu_{{\bf s}^2},k_2)
\end{aligned}
$$
If $n\in {\bf s}^2$, then we can compute in a similar way and deduce the same equality as above.

Since $f_n$ is an involution, we get that $$\begin{aligned}
&\quad\,\sum_{({\bf s}^1,{\bf s}^2)\in D^2(n)}\prod_{\substack{k_1=1,\cdots,|{\bf s}^1|,\\k_2=1,\cdots,|{\bf s}^2|}}N^{\Lam^1}(w_{{\bf s},1},\nu_{{\bf s}^1},k_1)N^{\Lam^2}(w_{{\bf s},2},\nu_{{\bf s}^2},k_2)\\
&=\frac{1}{2}\sum_{({\bf s}^1,{\bf s}^2)\in D^2(n)}\Big(\prod_{\substack{k_1=1,\cdots,|{\bf s}^1|,\\k_2=1,\cdots,|{\bf s}^2|}}N^{\Lam^1}(w_{{\bf s},1},\nu_{{\bf s}^1},k_1)N^{\Lam^2}(w_{{\bf s},2},\nu_{{\bf s}^2},k_2)+\\
&\qquad\qquad \prod_{\substack{k_1=1,\cdots,|{\bf \widetilde{s}}^1|,\\k_2=1,\cdots,|{\bf \widetilde{s}}^2|}}N^{\Lam^1}(w_{\widetilde{\bf s},1},\nu_{{\bf \widetilde{s}}^1},k_1)N^{\Lam^2}(w_{\widetilde{\bf s},2},\nu_{{\bf \widetilde{s}}^2},k_2)\Big)\\
&=\frac{1}{2}N^{\Lam}(w,\nu,n)\sum_{({\bf s}^1,{\bf s}^2)\in D^2(n)}\prod_{\substack{k_1=1,\cdots,|{\bf s}^1\setminus\{n\}|,\\k_2=1,\cdots,|{\bf s}^2\setminus\{n\}|}}N^{\Lam^1}(w_{{\bf s},1},\nu_{{\bf s}^1},k_1)N^{\Lam^2}(w_{{\bf s},2},\nu_{{\bf s}^2},k_2).
\end{aligned}
$$
Similarly, we can define $$f_{n-1}: D^2(n)\rightarrow D^2(n)$$
\[({\bf s}^1,{\bf s}^2)\mapsto
\begin{cases}
({\bf s}^1\setminus\{n-1\},{\bf s}^2\cup\{n-1\}), &\text{if $n-1\in {\bf s}^1$;}\\
({\bf s}^1\cup\{n-1\},{\bf s}^2\setminus\{n-1\}), &\text{if $n-1\in {\bf s}^2$,}
\end{cases}\]
where ${\bf s}^i\setminus\{n-1\}$ means we remove the integer $n-1$ from ${\bf s}^i$, and ${\bf s}^i\cup\{n-1\}$ means we inset the integer $n-1$ into ${\bf s}^i$ such that it is again in increasing order. We define $({\bf \hat{s}}^1,{\bf \hat{s}}^2):=f_{n-1}({\bf s}^1,{\bf s}^2)$.

It's easy to see $f_{n-1}$ is a well-defined bijection. Using the same argument as in the second last paragraph and the definition of $N^{\Lam}(w,\nu,n-1)$, we can deduce that $$\begin{aligned}
&\quad \prod_{\substack{k_1=1,\cdots,|{\bf s}^1\setminus\{n\}|,\\k_2=1,\cdots,|{\bf s}^2\setminus\{n\}|}}N^{\Lam^1}(w_{{\bf s},1},\nu_{{\bf s}^1},k_1)N^{\Lam^2}(w_{{\bf s},2},\nu_{{\bf s}^2},k_2)+\\
&\qquad\qquad\prod_{\substack{k_1=1,\cdots,|{\bf\hat{s}}^1\setminus\{n\}|,\\k_2=1,\cdots,|{\bf\hat{s}}^2\setminus\{n\}|}}N^{\Lam^1}(w_{\hat{\bf s},1},\nu_{{\bf \hat{s}}^1},k_1)N^{\Lam^2}(w_{\hat{\bf s},2},\nu_{{\bf \hat{s}}^2},k_2)\\
&=N^{\Lam}(w,\nu,n-1)\prod_{\substack{k_1=1,\cdots,|{\bf s}^1\setminus\{n-1,\,n\}|,\\k_2=1,\cdots,|{\bf s}^2\setminus\{n-1,
,n\}|}}N^{\Lam^1}(w_{{\bf s},1},\nu_{{\bf s}^1},k_1)N^{\Lam^2}(w_{{\bf s},2},\nu_{{\bf s}^2},k_2).
\end{aligned}
$$
Hence, we have:$$\begin{aligned}
&\quad\,\sum_{({\bf s}^1,{\bf s}^2)\in D^2(n)}\prod_{\substack{k_1=1,\cdots,|{\bf s}^1|,\\k_2=1,\cdots,|{\bf s}^2|}}N^{\Lam^1}(w_{{\bf s},1},\nu_{{\bf s}^1},k_1)N^{\Lam^2}(w_{{\bf s},2},\nu_{{\bf s}^2},k_2)\\
&=\frac{1}{2^2}N^{\Lam}(w,\nu,n-1)N^{\Lam}(w,\nu,n)\times\\
&\qquad \sum_{({\bf s}^1,{\bf s}^2)\in D^2(n)}\prod_{\substack{k_1=1,\cdots,|{\bf s}^1\setminus\{n-1,n\}|,\\k_2=1,\cdots,|{\bf s}^2\setminus\{n-1,n\}|}}N^{\Lam^1}(w_{{\bf s},1},\nu_{{\bf s}^1},k_1)N^{\Lam^2}(w_{{\bf s},2},\nu_{{\bf s}^2},k_2).
\end{aligned}
$$

Repeating the above argument with $n-1$ replaced by $n-2,n-3,\cdots,1$ and remember $|D^2(n)|=2^n$, we can get that
$$\begin{aligned}
&\quad\,\sum_{({\bf s}^1,{\bf s}^2)\in D^2(n)}\prod_{\substack{k_1=1,\cdots,|{\bf s}^1|,\\k_2=1,\cdots,|{\bf s}^2|}}N^{\Lam^1}(w_{{\bf s},1},\nu_{{\bf s}^1},k_1)N^{\Lam^2}(w_{{\bf s},2},\nu_{{\bf s}^2},k_2)\\
&=\frac{1}{2^n}N^{\Lam}(w,\nu,1)\cdots N^{\Lam}(w,\nu,n)\sum_{({\bf s}^1,{\bf s}^2)\in D^2(n)}1\\
&=N^{\Lam}(w,\nu,1)\cdots N^{\Lam}(w,\nu,n),
\end{aligned}
$$ which completes the proof of our claim (\ref{claimrd11}).
\end{proof}

Recall that for each $\beta=\sum_{i\in I}k_i\alpha_i\in Q^+$, $|\beta|=\sum_{i\in I}k_i$.

\begin{cor}\label{level2cor} Let $\mu\in I^\beta$, $\Lam=\Lam^1+\Lam^2$ with $\Lam^1,\Lam^2\in P^+$. Then
$$\begin{aligned}
\dim\RR^\Lam(\beta) e(\mu)=&\sum_{({\bf t}^1,{\bf t}^2)\in D^2(n)}\begin{pmatrix}|\beta|\\ |{\bf t}^1|\end{pmatrix} \dim\RR^{\Lam^1}(\beta_{\mu_{{\bf t}^1}})e(\mu_{{\bf t}^1})\times
\dim\RR^{\Lam^2}(\beta_{\mu_{{\bf t}^2}})e(\mu_{{\bf t}^2}),\\
\dim\RR^\Lam(\beta)=&\sum_{\substack{\beta_1,\beta_2\in Q^+\\ \beta=\beta_1+\beta_2}}\Bigl(\begin{matrix}|\beta|\\ |\beta_1|\end{matrix}\Bigr)^2\dim\RR^{\Lam^1}(\beta_1)\times\dim\RR^{\Lam^2}(\beta_2).
\end{aligned}$$
\end{cor}

\begin{proof} Applying Theorem \ref{mainthmC1}, we can get that $$\begin{aligned}
\dim\RR^\Lam(\beta)e(\mu)&=\sum_{\nu\in I^{\beta}}\sum_{\substack{\big(({\bf s}^1,{\bf s}^2),({\bf t}^1,{\bf t}^2)\big)\in D^2(\nu,\mu)}}\dim e(\nu_{{\bf s}^1})\RR^{\Lam^1}(\beta_{\nu_{{\bf s}^1}})e(\mu_{{\bf t}^1})\\
&\qquad\qquad\qquad \times\dim e(\nu_{{\bf s}^2})\RR^{\Lam^2}(\beta_{\nu_{{\bf s}^2}})e(\mu_{{\bf t}^2}).
\end{aligned}
$$
Note that the for any $\bi\in I^{\beta_{\mu_{{\bf t}^1}}}, \bj\in I^{\beta_{\mu_{{\bf t}^2}}}$, the number of triples $(\nu,{\bf s}^1,{\bf s}^2)$ such that
$({\bf s}^1,{\bf s}^2)\in D^2(n)$, $\nu\in I^\beta$, $\nu_{{\bf s}^1}=\bi$ and $\nu_{{\bf s}^2}=\bj$, is exactly $\begin{pmatrix}|\beta|\\ |{\bf s}^1|\end{pmatrix}=\begin{pmatrix}|\beta|\\ |{\bf t}^1|\end{pmatrix}$. Hence we get the first equation. The proof of the second equation is similar.
\end{proof}

Generalizing a little further, we call an $l$-tuple $\underline{k}=(k_1,\cdots,k_l)$ of non-negative integers a composition of $n$ with length $l$ if $k_1+\cdots+k_l=n$. We denote by $\mathcal{CP}_n^l$ the set of composition of $n$ with length $l$. For any $\underline{k}\in \mathcal{CP}_n^l$, we define $D^{\underline{k}}(n)$ to be the set of $\underline{k}=(k_1,\cdots,k_l)$-shuffles $({\bf s}^1,\cdots,{\bf s}^l)$ of $(1,2,\cdots,n)$. In particular, ${\bf s}^j$ is a strictly increasing sequence of $k_j$ integers for each $1\leq j\leq l$. Again, we allow some ${\bf s}^i$ to be empty. Now we define $$D^l(n)=:\bigsqcup_{\substack{\underline{k}\in \mathcal{CP}_n^l}}D^{\underline{k}}(n).$$ For any $\mu,\nu\in I^\beta$, we define $$
D^{l}(\nu,\mu):=\{\big(({\bf s}^1,\cdots,{\bf s}^l),({\bf t}^1,\cdots,{\bf t}^l)\big)\in D^{l}(n)\times D^l(n)|\beta_{\nu_{{\bf s}^i}}=\beta_{\mu_{{\bf t}^i}},\,\,i=1,\cdots,l\}.$$

\begin{cor}\label{genaralizetion}
Suppose $\Lam=\Lam^1+\cdots+\Lam^l$, where $\Lam^i\in P^+$ for each $1\leq i\leq l$. Then
$$\begin{aligned}
\dim e(\nu)\RR^\Lam e(\mu)&=\sum_{\substack{\big(({\bf s}^1,\cdots,{\bf s}^l),({\bf t}^1,\cdots,{\bf t}^l)\big)\in D^l(\nu,\mu)}} \dim e(\nu_{{\bf s}^1})\RR^{\Lam^1}(\beta_{\mu_{{\bf t}^1}})e(\mu_{{\bf t}^1})\times\cdots\\
&\qquad\qquad \times \dim e(\nu_{{\bf s}^l})\RR^{\Lam^l}(\beta_{\mu_{{\bf t}^l}})e(\mu_{{\bf t}^l})\\
\dim \RR^\Lam e(\mu)&=\sum_{({\bf t}^1,\cdots,{\bf t}^l)\in D^l(n)}\frac{(|{\bf t}^1|+\cdots+|{\bf t}^l|)!}{|{\bf t}^1|!\cdots|{\bf t}^l|!} \dim \RR^{\Lam^1}(\beta_{\mu_{{\bf t}^1}})e(\mu_{{\bf t}^1})\times\cdots\\
&\qquad\qquad\times\dim \RR^{\Lam^l}(\beta_{\mu_{{\bf t}^l}})e(\mu_{{\bf t}^l})
\end{aligned}$$
\end{cor}

\begin{proof} This follows from Theorem \ref{mainthmC1}, Corollary \ref{level2cor} and an induction on $l$.
\end{proof}

\medskip
\noindent
{\textbf{Proof of Theorem \ref{mainthmC}}}: This follows from Corollary \ref{genaralizetion} or induction on $l$ and using Corollary \ref{level2cor}.\qed
\medskip

\begin{rem}
The Level reduction formula does not hold for graded dimension.  For example, we consider $NH_1^2$, i.e. the cyclotomic nilHecke algebra. Then we have $$\begin{aligned}
\dim_q NH_1^2&=1+q^2\\
&\neq \dim_q NH_1^1 \dim_q NH_0^1+\dim_q NH_1^1 \dim_q NH_0^1=1+1.
\end{aligned}$$
\end{rem}

Corollary \ref{genaralizetion} and Theorem \ref{mainthmC} give us a way to compute the dimensions of higher level cyclotomic quiver Hecke algebras via the dimensions of some lower level (e.g., level $1$) cyclotomic quiver Hecke algebras. Using the combinatoric of shifted Young diagrams and Fock space realizations, Ariki and Park have given a dimension formula of finite quiver Hecke algebra (i.e., $\RR^{\Lam_0}(\beta)$) of type $A^{(2)}_{2k}$ in \cite[Theorem 3.4]{AP14}. Now using corollary \ref{genaralizetion}, we can generalize their combinatorial formula to $\RR^{l\Lam_0}(\beta)$, $l\in \N$ without Fock space realizations. Corollary \ref{genaralizetion} also sheds some light on the construction of higher level Fock spaces of arbitrary type via the tensor products of some level $1$ Fock spaces.

\begin{cor}\label{sufficient} Let $\Lam^i\in P^+, \beta_i\in Q^+$ for each $1\leq i\leq l$. Assume $\nu^i\in I^{\beta_i}$ and $e(\nu^i)\neq 0$ in $\RR^{\Lam^i}(\beta^i)$ for each $1\leq i\leq l$. Then $e(\nu)\neq 0$ in $\RR^{\Lam^1+\cdots+\Lam^l}(\beta_1+\cdots+\beta_l)$, for any $\nu\in\,{\rm Shuff}(\nu^1,\cdots,\nu^l)$, where ${\rm Shuff}(\nu^1,\cdots,\nu^l)$ means the set of all possible shuffles of $\nu^1,\cdots,\nu^l$.
\end{cor}

\begin{proof}
By assumption, $\dim e(\nu^1)\RR^{\Lam^1}(\beta_1)e(\nu^1)\cdots\dim e(\nu^l)\RR^{\Lam^l}(\beta_l)e(\nu^l)\neq 0$. Applying Corollary \ref{genaralizetion}, we deduce that $e(\nu)\neq 0$ in $\RR^{\Lam^1+\cdots+\Lam^l}(\beta_1+\cdots+\beta_l)$.
\end{proof}

\begin{cor}\label{necessary}
Suppose $e(\nu)\neq 0$ in $\RR^{\Lam}(\beta)$. Write $\Lam=\Lam^1+\cdots+\Lam^l$ to be a sum of $l$ dominant weights with lower levels. Then there exists $\nu^1,\cdots,\nu^l$, where $\nu^i\in I^{\beta_i},$ and $\beta_1+\cdots+\beta_l=\beta$, such that $e(\nu^i)\neq 0$ in $\RR^{\Lam^i}(\beta_i)$, $i=1,\cdots,l$ and $\nu$ is a shuffle of $\nu_1,\cdots,\nu_l$.
\end{cor}

\begin{proof} This follows directly from Corollary \ref{genaralizetion}.
\end{proof}

The following corollary gives a third criterion for $e(\nu)\neq 0$ in $\RR^\Lam(\beta)$. In type $A$ or type $C$, this follows from the Fock space realizations. Our result here is valid for {\it arbitrary} symmetrizable Cartan matrix.

\begin{cor}\label{maincorC3}
Let $\beta\in I^n, \nu\in I^\beta$. Assume $\Lam=\Lam_{t_1}+\cdots+\Lam_{t_l}$, where $t_i\in I$ for each $1\leq i\leq l$. Then $e(\nu)\neq 0$ in $\RR^\Lam_{\beta}$ if and only if $\nu$ is a shuffle of some $l$-tuple $(\nu^1,\nu^2,\cdots,\nu^l)$, such that $\beta=\beta_{\nu^1}+\cdots+\beta_{\nu^l}$, and $e(\nu^i)\neq 0$ in $\RR^{\Lam_{t_i}}(\beta_{\nu^i})$.
\end{cor}

\begin{proof}
The necessary part follows from Corollary \ref{necessary} and the sufficient part follows from Corollary \ref{sufficient}.
\end{proof}

\bigskip

\section{Monomial bases of $e(\wnu)\RR^\Lam(\beta)e(\mu)$ and $e(\mu)\RR^\Lam(\beta)e(\wnu)$}

Throughout this section, we fix $p\in\N$, $\fb:=(b_1,\cdots,b_p)\in\N^p$ and $\nu^1,\cdots,\nu^p\in I$ such that $\nu^i\neq\nu^j$ for any $1\leq i\neq j\leq p$ and $\sum_{i=1}^{p}b_i=n$. We define \begin{equation}\label{wnu}
\wnu=(\wnu_1,\cdots,\wnu_n):=\bigl(\underbrace{\nu^1,\cdots,\nu^1}_{\text{$b_1$ copies}},\cdots,\underbrace{\nu^p,\cdots,\nu^p}_{\text{$b_p$ copies}}\bigr)\in I^\beta ,
\end{equation}
where $\beta\in Q_n^+$. We call the $b_i$-tuple $(\underbrace{\nu^{i},\nu^{i},\cdots,\nu^{i}}_{b_{i}})$ the $i$th part of $\wnu$. As before, we set $b_0:=0, c_t:=\sum_{i=0}^{t}b_i$ for any $0\leq t\leq p$. The purpose of this section is to construct monomial bases for the subspaces $e(\wnu)\RR^\Lam(\beta)e(\mu)$ and $e(\mu)\RR^\Lam(\beta)e(\wnu)$ for arbitrary $\mu\in I^\beta$. In particular, we shall give the proof of our fourth main result Theorem \ref{mainthmD}.

\smallskip
\subsection{The case when $\mu=\wnu$}

The purpose of this section is to construct monomial bases for the subspace $e(\wnu)\RR^\Lam(\beta)e(\wnu)$.

\begin{dfn} For each $1\leq t\leq p$, we define $$
N^\Lam_t(\wnu):=N^\Lam(1,\wnu,c_{t-1}+1).
$$
\end{dfn}

Our assumption that $\nu^i\neq\nu^j$ for any $1\leq i\neq j\leq p$ implies that $\Sym(\wnu,\wnu)$ is the standard Young subgroup
$\Sym_{\fb}:=\Sym_{\{1,\cdots,c_1\}}\times\cdots\times\Sym_{\{c_{p-1}+1,\cdots,n\}}$ of $\Sym_n$. Moreover, since $\nu^t\neq\nu^j$ for any $1\leq j<t$, it follows from the original definition (\ref{Ndef}) that \begin{equation}\label{positive}
N^\Lam_t(\wnu)\geq 0,\quad\forall\,1\leq t\leq p.
\end{equation}




%

\begin{thm}\label{mainthm2a}  Let $\Lam\in P^+$ be arbitrary. Let $\beta\in Q_n^+$ such that $\wnu\in I^\beta$. Then we have
$$ \dim e(\wnu)\RR^{\Lam}(\beta)e(\wnu)=\prod_{i=1}^{p}\Bigl(b_{i}!\prod_{j=0}^{b_i-1}(N^\Lam_{i}(\wnu)-j)\Bigr). $$
In particular, $e(\wnu)\neq 0$ if and only if $N^\Lam_{i}(\wnu)\geq b_{i}$ for any $1\leq i\leq p$.
\end{thm}

\begin{proof}
The first part of the theorem follows from  Theorem \ref{mainthm1b}.

We now consider the second part. If $N^\Lam_{i}(\wnu)\geq b_{i}$ for any $1\leq i\leq p$, then by the first part of the theorem we have $\dim e(\wnu)\RR^{\Lam}(\beta)e(\wnu)>0$. In particular, $e(\wnu)\neq 0$. Conversely, suppose that
$N^\Lam_{i}(\wnu)\leq b_{i}-1$ for some $1\leq i\leq p$. By (\ref{positive}), $N^\Lam_i(\wnu)\geq 0$ for any $1\leq i\leq p$. It follows that $0$ must appears as a factor in the product $\prod_{j=0}^{b_i-1}(N^\Lam_{i}(\wnu)-j)$. Hence $\dim e(\wnu)\RR^{\Lam}(\beta)e(\wnu)=0$, which implies that $e(\wnu)=0$. This completes the proof of the second part and hence the whole theorem.
\end{proof}



Let $1\leq a<n$. Following \cite[(3.6)]{KK}, we define the operator $\partial_{a}$ on $$
\bigoplus_{\mu\in I^\beta}K[x_1,\cdots,x_n]e(\mu)\subset\RR(\beta)$$ by $$
\partial_a f:=\frac{s_a(f)-f}{x_a-x_{a+1}}\sum_{\substack{\mu\in I^\beta\\ \mu_a=\mu_{a+1}}}e(\mu),\,\,\, \forall\,f\in K[x_{1},x_{2},\cdots,x_{n}]e(\mu).
$$

\begin{lem}\label{Lemma 8}
Let $\beta\in Q_n^+$, $f\in K[x_{1},x_{2},\cdots,x_{n}]$, and $\nu\in I^{\beta}$ such that $\nu_{k}=\nu_{k+1}$, where $1\leq k<n$. If we have $fe(\nu)=0$ in $\RR^{\Lam}(\beta)$, then $\partial_{k}(f)e(\nu)=0$ in $\RR^{\Lam}(\beta)$.
\end{lem}

\begin{proof} This follows from \cite[Lemma 4.2]{KK} by taking $M=\RR^\Lam(\beta)$ there.
\end{proof}

\begin{lem}\label{Lemma 9} Let $p_1:=a^{\Lambda}_{\nu^{1}}(x_{1})$. For any $1<i\leq p$, we set $$
p_{c_{i-1}+1}=a^{\Lambda}_{\nu^{i}}(x_{c_{i-1}+1})\prod_{\substack{t=1}}^{i-1}\prod_{d=c_{t-1}+1}^{c_{t}}Q_{\nu^{t},\nu^{i}}(x_{d},x_{c_{i-1}+1}). $$
Then $p_{c_{i-1}+1}\in  \RR^\Lam(\beta)$ is a polynomial in $x_{c_{i-1}+1}$ of degree $N_{i}(\wnu)$ with leading coefficient in $K^\times$ and other coefficients in $K[x_{1},x_{2},\cdots,x_{c_{i-1}}]$. Moreover, $p_{c_{i-1}+1}e(\wnu)$ is a zero element in $e(\wnu)\RR^\Lam({\beta})e(\wnu)$.
\end{lem}

\begin{proof}
The first part is a direct computation. For the last part, just consider
$\psi_{c_{i-1}}\psi_{c_{i-1}-1}\cdots\psi_{1}a^{\Lambda}_{\nu^{i}}(x_{1})e(\widehat{\nu})\psi_{1}\psi_{2}\cdots\psi_{c_{i-1}}$, where $\widehat{\nu}$ is the $n$-tuple obtained by moving the $(c_{i-1}+1)$-th component of $\wnu$ (which is exactly $\nu^{i}$) to the first part and unchanging the relative positions of all the other components. By definition $a^{\Lambda}_{\nu^{i}}(x_{1})e(\widehat{\nu})=0$ in $\RR^\Lam(\beta)$. On the other hand, since $\nu^i\neq\nu^t$ for any $1\leq t<i$, we have that $$
\psi_{c_{i-1}}\psi_{c_{i-1}-1}\cdots\psi_{1}a^{\Lambda}_{\nu^{i}}(x_{1})e(\widehat{\nu})=
a^{\Lambda}_{\nu^{i}}(x_{c_{i-1}+1})\psi_{c_{i-1}}\psi_{c_{i-1}-1}\cdots\psi_{1}e(\widehat{\nu}).
$$
Finally, the lemma follows because $$
\psi_{c_{i-1}}\psi_{c_{i-1}-1}\cdots\psi_{1}e(\widehat{\nu})\psi_{1}\psi_{2}\cdots\psi_{c_{i-1}}
=\prod_{\substack{t=1}}^{i-1}\prod_{d=c_{t-1}+1}^{c_{t}}Q_{\nu^{t},\nu^{i}}(x_{d},x_{c_{i-1}+1}),
$$
where again we have used the assumption that $\nu^i\neq\nu^t$ for any $1\leq t<i$.
\end{proof}

\begin{prop}\label{Prop 2} Let $1\leq i\leq p$. For any integer $k$ which satisfies $c_{i-1}<k\leq c_{i}$, there exists a monic polynomial $p_{k}$ in $x_{k}$ of degree $N_{i}(\wnu)-(k-c_{i-1}-1)$ with coefficients in $K[x_{1},x_{2},\cdots,x_{k-1}]$. Moreover,
$p_{k}e(\wnu)$ is a zero element in $e(\wnu)\RR^\Lam({\beta})e(\wnu)$.
\end{prop}

\begin{proof}
By Lemma $\ref{Lemma 9}$, we see that, up to a scalar in $K^\times$, $p_{c_{i-1}+1}$ satisfies the requirement for $k=c_{i-1}+1$. We take $p_{c_{i-1}+2}=\partial_{c_{i-1}+1}(f)e(\wnu)$. Then by Lemma $\ref{Lemma 8}$, it's easy to see that $p_{c_{i-1}+2}$ also satisfies the requirement for $k=c_{i-1}+2$. In general, the proposition follows from an induction on $k$.
\end{proof}

\begin{thm}\label{mainthm2b} The following set \begin{equation}\label{base1}
\Bigl\{\psi_{w }\prod_{k=1}^{n}x_{k}^{r_{k}}e(\wnu)\Bigm|\begin{matrix}\text{$w\in\Sym_{\fb}$, for any $1\leq i\leq p$, $c_{i-1}<k\leq c_{i}$,}\\
\text{$r_{k}\in\{0,1,\cdots,N^\Lam_{i}(\wnu)-(k-c_{i-1})\}$}\end{matrix}\Bigr\}
\end{equation} forms a $K$-basis of $e(\wnu)\RR^\Lam({\beta})e(\wnu)$.
\end{thm}

\begin{proof}
Applying Proposition $\ref{Prop 2}$, we see that the elements in the above set (\ref{base1}) span the $K$-linear space $e(\wnu)\RR^\Lam({\beta})e(\wnu)$. Counting the dimensions and using Theorem \ref{mainthm2a}, we see the set (\ref{base1}) must be a $K$-basis of $e(\wnu)\RR^\Lam({\beta})e(\wnu)$. This proves the theorem.
\end{proof}

\begin{cor} We have that $$
\dim_q e(\wnu)\RR^\Lam({\beta})e(\wnu)=\prod_{i=1}^{p}\Bigl(\prod_{k=1}^{b_i}\frac{q_{\nu_k}^{-2k}-1}{q_{\nu_k}^{-2}-1}
\prod_{t=c_{i-1}+1}^{c_i}(1+q_{\nu_t}^{2}+\cdots+q_{\nu_t}^{2(N^\Lam_i(\wnu)-t)})\Bigr).
$$
\end{cor}

\begin{proof} This follows from Theorem \ref{mainthm2b}.
\end{proof}

\begin{prop}\label{mainprop2}
There is a $K$-linear isomorphism: $$\gamma:\,\,e(\wnu)\RR^\Lam({\beta})e(\wnu)\,\,\cong\,\,\mathscr{H}^{(0)}_{N^\Lam_{1}(\wnu),b_{1}}\otimes\mathscr{H}^{(0)}_{N^\Lam_{2}(\wnu),b_{2}}\otimes\cdots
\otimes\mathscr{H}^{(0)}_{N^\Lam_{p}(\wnu),b_{p}}. $$
\end{prop}

\begin{proof} For each $1\leq k\leq p$, we use $\tau_k$ to denote the canonical isomorphism $\Sym_{\{c_{k-1}+1,c_{k-1}+2,\cdots,c_k\}}\cong\Sym_{b_k}$ which is uniquely determined on generators by $s_{c_{k-1}+j}\mapsto s_j,\,\forall\,\leq j<b_k$. We construct a linear map  $$\gamma:\,\,e(\wnu)\RR^\Lam({\beta})e(\wnu)\,\,\rightarrow\,\,\mathscr{H}^{(0)}_{N^\Lam_{1}(\wnu),b_{1}}\otimes\mathscr{H}^{(0)}_{N^\Lam_{2}(\wnu),b_{2}}\otimes\cdots\otimes
\mathscr{H}^{(0)}_{N^\Lam_{p}(\wnu),b_{p}} $$
which sends $\psi_{u_{1}u_{2}\cdots u_{p}}\prod_{k=1}^{n}x_{k}^{r_{k}}e(\wnu)$ to $$
(\psi_{\tau_1(u_1)}X_{1},\psi_{\tau_2(u_2)}X_{2},\cdots,\psi_{\tau_p(u_p)}X_{p}),
$$
where for each $1\leq i\leq p$, $u_i\in\Sym_{\{c_{i-1}+1,\cdots,c_i\}}$ and $X_{i}:=\prod_{k=1}^{b_{i}}x^{r_{k+c_{i-1}}}_{k}$,
and for each $c_{i-1}+1\leq t\leq c_i$, $r_{t}\in\{0,1,\cdots,N_{i}(\wnu)-(t-c_{i-1})\}$. Applying \cite[Theorem 2.34]{HuL}, Theorem \ref{mainthm2b} and Corollary \ref{maincor2a}, one sees that $\gamma$ is a $K$-linear isomorphism.
\end{proof}



\smallskip
\subsection{The general case}

In this subsection we shall construct monomial bases for the subspaces $e(\wnu)\RR^\Lam(\beta)e(\mu)$ and $e(\mu)\RR^\Lam(\beta)e(\wnu)$ for arbitrary $\mu\in I^\beta$.

Recall that we have fixed a special $n$-tuple $\wnu\in I^n$ at the beginning (\ref{wnu}) of this section. Let $\beta\in Q_n^+$ such that $\wnu\in I^\beta$. For any $\mu\in I^\beta$, we can always choose a minimal length right $\Sym_\fb$-coset representative $d_\mu$ of $\Sym_{\fb}$ in $\Sym_n$ such that $d_\mu^{-1}\wnu=\mu$. In particular, $\Sym(\wnu,\mu)=d_{\mu}^{-1}\Sym_{\fb}$ and hence $\Sym(\mu,\wnu)=\Sym_{\fb}d_{\mu}$.

The following crucial definition plays an important role in our later construction of  monomial bases for the subspaces $e(\wnu)\RR^\Lam(\beta)e(\mu)$ and $e(\mu)\RR^\Lam(\beta)e(\wnu)$.

\begin{dfn}\label{weight} Let $\mu=(\mu_1,\cdots,\mu_n)\in I^\beta, 1\leq k\leq n$. We define \begin{equation}\label{wt}
N^\Lam(\mu,k):=N^\Lam(d_\mu,\mu,k)+|\{1\leq j<k|\mu_{j}=\mu_{k}\}|. \end{equation}
\end{dfn}

\begin{examp}\label{toyexamp}
Suppose $\mu=\wnu$, then $d_\mu=1$ and $N^\Lam(\mu,k)=N_{i}(\wnu)-(k-c_{i-1}-1)$ whenever $c_{i-1}<k\leq c_{i}$ for some $1\leq i\leq p$.
\end{examp}

The following result is a crucial ingredient in the proof of our main result in this subsection.
\begin{lem}\label{keylem2} Let $1\leq i\leq p$ and $\mu\in I^\beta$. Let $1\leq t_1<t_2<\cdots<t_{b_i}\leq n$ be the unique $b_i$ integers such that $\mu_{t_j}=\nu^i$.
Let $w=w_1\times\cdots\times w_{p}\in\Sym_\fb$, where $w_{k}\in\Sym_{\{c_{k-1}+1,\cdots,c_{k}\}}$ for each $1\leq k\leq p$. Then for any $1\leq j\leq b_i$, $$
N^\Lam(wd_\mu,\mu,t_{j})=N^\Lam(d_\mu,\mu,t_{j})+2(j-1)-2|\tilde{J}_{w_i}^{<d_\mu(t_j)}|, $$
where $\tilde{J}_{w_i}^{<d_\mu(t_j)}:=\{c_{i-1}+1\leq a<d_\mu(t_j)|w_i(a)<w_i(d_\mu(t_j))\}$.
In particular, $N^\Lam(wd_\mu,\mu,t_{j})$ does not depend on $w_k$ for $1\leq k\neq i\leq p$.
\end{lem}

\begin{proof}
By definition of $d_\mu\in\Sym(\mu,\wnu)$, $d_{\mu}(k)\in \{c_{r-1}+1,c_{r-1}+2,\cdots,c_{r}\}$ whenever $\mu_{k}=\nu^{r}$. Therefore, we have  $$\begin{aligned}
J_{wd_{\mu}}^{<t_j}&=\{1\leq s<t_{j}|wd_{\mu}(s)<wd_{\mu}(t_{j})\}\\
&=\bigl\{1\leq s<t_j\bigm|s\notin\{t_{1},t_{2},\cdots t_{j-1}\}, wd_{\mu}(s)<wd_{\mu}(t_{j})\bigr\}\\
&\qquad \cup \bigl\{t_{a}\bigm|1\leq a\leq j-1,\,\,wd_{\mu}(t_a)<w d_{\mu}(t_{j})\bigr\}\\
&=\bigl\{1\leq s<t_j\bigr|s\notin\{t_{1},t_{2},\cdots t_{j-1}\}, d_{\mu}(s)<d_{\mu}(t_{j})\bigr\}\\
&\qquad\cup \bigl\{t_{a}\bigm|1\leq a\leq j-1,\,\,w_{i} d_{\mu}(t_{a})<w_{i} d_{\mu}(t_{j})\bigr\}.
\end{aligned}
$$

Since $d_{\mu}$ is a minimal length right $\Sym_\fb$-coset representative in $\Sym_n$, we have $d_{\mu}(t_{1})<d_{\mu}(t_{2})<\cdots<d_{\mu}(t_{b_{i}})$.
It follows that $$\begin{aligned}
N^\Lam(wd_{\mu},\mu,t_{j})&=(\Lambda-\sum_{s\in J_{wd_{\mu}}^{<t_j}}\alpha_{\mu_{s}}\big )(h_{\mu_{t_{j}}})\\
&=N^\Lam(d_{\mu},\mu,t_{j})+2(j-1)-2|\tilde{J}_{w_{i}}^{<d_{\mu}(t_j)}|.
\end{aligned}
$$
This completes the proof of the lemma.
\end{proof}

\begin{thm}\label{mainthm3a} Let $\mu=(\mu_1,\cdots,\mu_n)\in I^\beta$. Then we have $$
\dim\,e(\wnu)\RR^\Lam({\beta})e(\mu)=\dim\,e(\mu)\RR^\Lam({\beta})e(\wnu)=\Bigl(\prod_{i=1}^{p}b_{i}!\Bigr)\Bigl(\prod_{t=1}^{n}N^\Lam(\mu,t)\Bigr).
$$
\end{thm}

\begin{proof} Using the anti-isomorphism $\ast$, we see that $$
\dim\,e(\wnu)\RR^\Lam({\beta})e(\mu)=\dim\,e(\mu)\RR^\Lam({\beta})e(\wnu) .
$$

Note that $\Sym(\mu,\wnu)=\Sym_\fb d_\mu$. Applying Theorem \ref{mainthmA} and Lemma \ref{keylem2}, we have $$\begin{aligned}
&\quad\,\dim\,e(\mu)\RR^\Lam({\beta})e(\wnu)\\
&=\sum_{w\in\Sym_{\fb}}\prod_{t=1}^{n}N^\Lam(wd _{\mu},\mu,t)\\
&=\prod_{i=1}^{p}\sum_{u\in\Sym_{\{c_{i-1}+1,\cdots,c_i\}}}\prod_{\substack{1\leq t\leq n\\ \mu_{t}=\nu^i}}N^\Lam(ud_{\mu},\mu,t).
\end{aligned}$$

For each $1\leq i\leq p$, we denote by $1\leq t_{i1}<t_{i2}<\cdots<t_{ib_i}\leq n$ the unique $b_i$-tuple such that $\mu_{t_{ij}}=\nu^i$, $\forall\,1\leq j\leq b_i$.
For each $1\leq j\leq b_i$, we set $$
N_{ij}:=N^\Lam(d_\mu,\mu,t_{ij})+2(j-1). $$
Then, using Lemma \ref{keylem2} again, combing with the bijection in Lemma \ref{bij}, we can deduce that
$$\begin{aligned}
&\quad\,\sum_{u\in\Sym_{\{c_{i-1}+1,\cdots,c_i\}}}\prod_{\substack{1\leq t\leq n\\ \mu_t=\nu^i}} N^\Lam(ud_{\mu},\mu,t)\\
&=\prod_{k=1}^{b_i}((N_{ik}+N_{ik}-2+N_{ik}-4+\cdots+N_{ik}-2(k-1))\\
&={b_i}!\prod_{k=1}^{b_i}(N_{ik}-(k-1))=b_{i}!\prod_{\substack{1\leq t\leq n\\ \mu_{t}=\nu^{i}}}N^\Lam(\mu,t).
\end{aligned}
$$
Finally, we consider the products of the above identity over $1\leq i\leq p$. Then we can deduce that $\dim\,e(\mu)\RR^\Lam({\beta})e(\wnu)=\Bigl(\prod_{i=1}^{p}b_{i}!\Bigr)\Bigl(\prod_{t=1}^{n}N^\Lam(\mu,t)\Bigr)$.
This completes the proof of the theorem.
\end{proof}

\begin{cor}\label{maincor3} Let $\mu\in I^\beta$. Then $e(\wnu)\RR^\Lam({\beta})e(\mu)\neq 0$ if and only if for any $1\leq k\leq n$, $N^\Lam(\mu,k)>0$.
\end{cor}

\begin{proof} The ``if'' part of the corollary follows directly from Theorem \ref{mainthm3a}. It remains to prove the ``only if'' part of the corollary.

Suppose that $e(\wnu)\RR^\Lam({\beta})e(\mu)\neq 0$. Assume there exists some $1\leq s\leq n$ such that $N^\Lam(\mu,s)\leq 0$. First, $e(\wnu)\RR^\Lam({\beta})e(\mu)\neq 0$ implies that for any $1\leq k\leq n$, $N^\Lam(\mu,k)\neq 0$.

For each $1\leq i\leq p$ we define $$
\{t_{ia}|1\leq a\leq b_i, t_{i1}<t_{i2}<\cdots<t_{ib_i}\}:=\{1\leq k\leq n|\mu_k=\nu^i\}.
$$
By definition (because $a_{kl}\leq 0$ for any $k\neq l$)), $N_\mu^\Lam(t_{i1})>0$ for any $1\leq i\leq p$.

Suppose that $N^\Lam(\mu,t_{ij})<0$ for some $1\leq j\leq b_i$ and $1\leq i\leq p$. Assume that $i,j$ is chosen such that $t_{ij}$ is as minimal as possible. By the last paragraph, we can deduce that $j>1$. Thus $N^\Lam(\mu,t_{ia})>0$ for any $1\leq a<j$.
Note that $d_\mu(t_{i(j-1)})<d_\mu(t_{ij})$ and $\<\alpha_{\mu_{t_{i(j-1)}}},h_{\mu_{t_{ij}}}\>=2$. It follows that $$
N^\Lam(\mu,t_{i(j-1)})\leq N^\Lam(\mu,t_{ij})+1 ,
$$
which is a contradiction because $N^\Lam(\mu,t_{ij})<0<N^\Lam(\mu,t_{i(j-1)})$. This completes the proof of the ``only if'' part and hence the corollary.
\end{proof}

We want to construct an explicit homogeneous monomial bases for $e(\wnu)\RR^\Lam({\beta})e(\mu)$ and $e(\mu)\RR^\Lam({\beta})e(\wnu)$, from which
one can also derive the graded dimensions of these two subspaces.

\begin{lem}\label{psidmu} Let $\mu\in I^\beta$. Let $s_{i_1}\cdots s_{i_m}$ and $s_{j_1}\cdots s_{j_m}$ be two reduced expression of $d_\mu$. Then $$
\psi_{i_1}\cdots \psi_{i_m}e(\mu)=\psi_{j_1}\cdots \psi_{j_m}e(\mu) .
$$
In other words, $\psi_{d_\mu}e(\mu):=\psi_{i_1}\cdots \psi_{i_m}e(\mu)$ depends only on $\mu$ but not on the choices of the reduced expression of $d_\mu$.
\end{lem}

\begin{proof} Applying the defining relation of $\RR^\Lam(\beta)$ or \cite[Theorem 4.10]{BKW}, we see that $\psi_{i_1}\cdots \psi_{i_m}e(\mu)-\psi_{j_1}\cdots \psi_{j_m}e(\mu)$ is either equal zero or equal to a $K$-linear combination of some elements of the form $$
e(\wnu)\psi_{p_1}\cdots \psi_{p_t}x_1^{d_1}\cdots x_n^{d_n}e(\mu),
$$
where $t<m$, $d_1,\cdots d_n\in\N$. However, $d_\mu$ is a minimal length right $\Sym_\fb$-coset representative in $\Sym_n$ such that $d_\mu\mu=\wnu$, which is
a minimal length element in $\Sym_n$ such that $d_\mu\mu=\wnu$. It follows that the second case can not happen. In other words, $\psi_{i_1}\cdots \psi_{i_m}e(\mu)=\psi_{j_1}\cdots \psi_{j_m}e(\mu)$.
%
\end{proof}

\begin{lem}\label{keylem3} Let $\mu\in I^\beta$. Suppose that $1\leq k\leq n$ with $N^\Lam(\mu,k)>0$. Then there exists a monic polynomial $p_{k}$ in $x_{k}$ of degree $N^\Lam(\mu,k)$ with coefficients in $K[x_{1},x_{2},\cdots,x_{k-1}]$. Moreover, $\psi_{d_\mu} p_{k}e(\mu)$ is a zero element in $e(\wnu)\RR^\Lam({\beta})e(\mu)$.
\end{lem}
\begin{proof}
Suppose $\mu_k=\nu^i$, where $\nu^i\in I$. In particular, $c_{i-1}<d_\mu(k)\leq c_i$. Recall the definitions of $\wnu$ and $\{c_j|1\leq j\leq p\}$ at the beginning of this section. We define $\mathcal{J}_i:=\{1\leq m<k|d_\mu(m)>c_{i}\}$ and write $$
\mathcal{J}_i=\{m_j|1\leq j\leq g, 1\leq m_{1}<m_{2}<\cdots<m_{g}<k\}.
$$
Then $\mathcal{J}_i=\{1\leq m\leq k|\mu_m=\nu^t, i<t\leq p\}$.

We consider the following products of cycles: $$\begin{aligned}
u_1:&=(k-g+1,k-g,\cdots,m_1+1,m_1)(k-g+2,k-g+1,\cdots,m_2+1,m_2)\cdots \\
&\qquad\qquad (k,k-1,\cdots,m_g+1,m_g) .\end{aligned}
$$
Clearly we have $$
u_1:=(s_{k-g}\cdots s_{m_{1}+1}s_{m_{1}})(s_{k-g+1}\cdots s_{m_2+1}s_{m_2})\cdots (s_{k-1}\cdots s_{m_g+1}s_{m_g}),
$$
and this is a reduced expression of $u_1$. We set $\mu^{[1]}:=u_1\mu$. In other words, $\mu^{[1]}$ is obtained from $\mu$ by moving its $m_1$-th, $\cdots$ $m_g$-th components to the $(k-g+1)$-th, $\cdots$, $k$-th positions respectively, and unchanging the relative positions of all the remaining components of $\mu$.
In particular, we have $\mu^{[1]}_{k-g}=\nu^i$ and there is no $t<k-g$ such that $\mu^{[1]}_{t}=\nu^j$ with $j>i$.

Now we define $\mathcal{J}'_i:=\{1\leq l<k-g|\mu_l^{[1]}=\nu^{i}\}$ and write $$
\mathcal{J}'_i=\{l_i|1\leq i\leq r, 1\leq l_{1}<l_{2}<\cdots<l_{r}<k-g\}.
$$
Let $\mu^{[2]}$ be the $n$-tuple obtained from $\mu^{[1]}$ by moving its $l_1$-th, $\cdots$ $l_r$-th components to the $(k-g-r)$-th, $\cdots$, $(k-g-1)$-th positions respectively, and unchanging the relative positions of all the remaining components of $\mu^{[1]}$. In fact, we can choose $u_2$ to be the unique minimal element satisfying $\mu^{[2]}=u_2\mu^{[1]}$. In particular, for any $a<k-g-r$ we have $\mu^{[2]}_a=\nu^j$ with $j<i$; while for any $k-g-r\leq b\leq k-g$ we have $\mu^{[2]}_a=\nu^i$.

Let $\widehat{\mu}$ be the $n$-tuple obtained from $\mu^{[2]}$ by moving the $(k-g-r)$-th component $\mu^{[2]}_{k-g-r}$ (which is equal to $\nu^{i}$ by construction) of $\mu^{[2]}$  to the first position and unchanging the relative positions of all the other components. We consider $$\psi_{k-g-r}\psi_{k-g-r-1}\cdots\psi_2\psi_{1}a^{\Lambda}_{\nu^{i}}(x_{1})e(\widehat{\mu})\psi_{1}\psi_{2}\cdots\psi_{k-g-r-1}\psi_{k-g-r}.
$$
The same argument as in the proof of Lemma \ref{Lemma 9} shows that this equals to $\widehat{p}_k e(\mu^{[2]})$, where $\widehat{p}_k$ is a polynomial in $x_{k-g-r}$ of degree $N^\Lam(\mu,k)+r$ with leading coefficient in $K^\times$ and other coefficients in $K[x_{1},x_{2},\cdots,x_{k-g-r-1}]$. Clearly, this is zero in $\RR^\Lam({\beta})e(\mu^{[2]})$.

Using Lemma \ref{Lemma 8} we can deduce that there is a monic polynomial $p_k^{[2]}$ in $x_{k-g}$ of degree $N^\Lam(\mu,k)$ with coefficients in $K[x_{1},x_{2},\cdots,x_{k-g-1}]$, and satisfies that $p_k^{[2]}e(\mu^{[2]})$ is zero $\RR^\Lam({\beta})e(\mu^{[2]})$. Now we define $p_{k}=u_1^{-1}u_2^{-1}(p_k^{[2]})$, then $p_{k}$ is a monic polynomial in $x_{k}$ of degree $N^\Lam(\mu,k)$ with coefficients in $K[x_{1},x_{2},\cdots,x_{k-1}]$ and $$\begin{aligned}\label{poly}
\psi_{u_2}\psi_{u_1}p_k e(\mu)=p_k^{[2]}\psi_{u_2}\psi_{u_1}e(\mu)=p_k^{[2]}e(\mu^{[2]})\psi_{u_2}\psi_{u_1}=0.
\end{aligned}
$$
Finally, by construction we can find $u_3\in\Sym_n$ such that $d_\mu=u_3u_2u_1$, and $\ell(d_\mu)=\ell(u_3)+\ell(u_2)+\ell(u_1)$. Hence by Lemma \ref{psidmu}, $\psi_{d_\mu} p_{k}e(\mu)=\psi_{u_3}\psi_{u_2}\psi_{u_1}p_k e(\mu)=0$.
\end{proof}

Henceforth, for each $w\in\Sym_\fb$, we fix a reduced expression $s_{j_1}\cdots s_{j_a}$ of $w$ and define \begin{equation}\label{psiwd1}
\psi^{\bf 1}_{wd_\mu}:=\psi_{j_1}\cdots\psi_{j_a}\psi_{d_\mu} .
\end{equation}
Note that every element in $\Sym(\mu,\wnu)$ is of the form $wd_\mu$ for some $w\in\Sym_{\fb}$.

\begin{thm}\label{mainthm3b}
Suppose that $N^\Lam(\mu,k)>0$ for any $1\leq k\leq n$. Then the elements in the following set $$
\Bigl\{\psi^{\bf 1}_{w}\prod_{k=1}^{n}x_{k}^{r_{k}}e(\mu)\Bigm| w \in\Sym(\mu,\wnu), 0\leq r_{k}<N^\Lam(\mu,k), \forall\,1\leq k\leq n\Bigr\}$$
form a $K$-basis of $e(\wnu)\RR^\Lam({\beta})e(\mu)$.
\end{thm}
\begin{proof}
This follows from Theorem \ref{mainthm3a} and Lemma \ref{keylem3}.
\end{proof}

\medskip
\noindent
{\bf Proof of Theorem \ref{mainthmD}}: For each $j>0$, we define $$M_j=\text{$K$-Span}\Bigl\{\psi^{\bf 1}_{w}\prod_{k=1}^{n}x_{k}^{r_{k}}e(\mu)\Bigm| w \in\Sym(\mu,\wnu),\, \ell(w)<j,\,0\leq r_{k}<N^\Lam(\mu,k), \forall\,1\leq k\leq n\Bigr\}.$$ We claim that for any $w \in\Sym(\mu,\wnu)$, any reduced expression $w=s_{i_1}\cdots s_{i_t}$ of $w$ and any non-negative integers $\{r_k\geq 0|1\leq k\leq n\}$, \begin{equation}\label{unitriangular}
\psi_{i_1}\cdots\psi_{i_t}\prod_{k=1}^{n}x_{k}^{r_{k}}e(\mu)-\psi^{\bf 1}_{w}\prod_{k=1}^{n}x_{k}^{r_{k}}e(\mu)\in M_{\ell(w)}.
\end{equation}
We prove this by induction on $\ell(w)$. When $w=d_\mu$ this follows from Lemma \ref{psidmu}. As in Lemma \ref{psidmu}, we can write $\psi_{i_1}\cdots\psi_{i_t}\prod_{k=1}^{n}x_{k}^{r_{k}}e(\mu)-\psi^{\bf 1}_{w}\prod_{k=1}^{n}x_{k}^{r_{k}}e(\mu)$ as a $K$-linear combination of some elements of the form $$
\psi_{p_1}\cdots \psi_{p_s}x_1^{d_1}\cdots x_n^{d_n}e(\mu),
$$
where $s<\ell(w)$, $d_1,\cdots d_n\in\N$ and $s_{p_1}\cdots s_{p_s}$ is a reduced expression of $u:=s_{p_1}\cdots s_{p_s}$. Then by induction hypothesis, we have $$\psi_{p_1}\cdots \psi_{p_s}x_1^{d_1}\cdots x_n^{d_n}e(\mu)\in \psi^{\bf 1}_u x_1^{d_1}\cdots x_n^{d_n}e(\mu)+M_{\ell(u)}.
$$ Now applying Lemma \ref{keylem3}, we can see $\psi^{\bf 1}_u x_1^{d_1}\cdots x_n^{d_n}e(\mu)\in M_{\ell(u)+1}\subseteq M_{\ell(w)}$. Moreover, $M_{\ell(u)}\subset M_{\ell(w)}$. Hence our claim follows. Since the transition matrix between the elements given in Theorem \ref{mainthm3b} and the elements given in Theorem \ref{mainthmD} is unitriangular, Theorem \ref{mainthmD} follows from Theorem \ref{mainthm3b} immediately.
\qed
\medskip

Using the anti-isomorphism $\ast$ of $\RR^\Lam(\beta)$, one can also get a $K$-basis for the subspace $e(\mu)\RR^\Lam({\beta})e(\wnu)$. Next we want to compare two different such kind of spaces.

\begin{lem}\label{Lemma 12} Let $\mu\in I^n$ and $1\leq k<n$. If $d _{\mu}>d _{\mu}s_k$, then $d _{\mu s_{k}}=d _{\mu}s_{k}$. In general, if
$d _{\mu}=d_{1}d_{2}$, with $\ell(d _{\mu})=\ell(d_{1})+\ell(d_{2})$, then $d _{\mu d_{2}^{-1}}=d_{1}$.
\end{lem}

\begin{proof} This follows from \cite[Lemma 1.4(ii)]{DJ1}.
\end{proof}

\begin{lem}\label{Lemma 13} Let $1\leq a<n$. Suppose that $d_{\mu}>d_{\mu}s_a$ (and hence $d_\mu \(a)>d_\mu(a+1)$), then
\[N^\Lam(\mu,k)=
\begin{cases}
N^\Lam(\mu s_{a},k),& \text{if $k\neq a,\,a+1$};\\
N^\Lam(\mu s_{a}, k+1)+\<\alpha_{\mu_{a+1}},h_{\mu_{a}}\>, & \text{if $k=a$}; \\
N^\Lam(\mu s_{a}, k-1), & \text{if $k=a+1$}.\\
\end{cases}\]
\end{lem}

\begin{proof} Suppose $k\neq a,a+1$. We consider the map $$
\theta_a: J_{d_\mu}^{<k}\rightarrow J_{d_\mu s_a}^{<k},\,\, t\mapsto s_a(t) .
$$
It is clear that $\theta_a$ is a well-defined bijection in this case. Thus $N^\Lam(\mu s_{a},k)=N^\Lam(\mu,k)$.

Suppose $k=a+1$. Then in this case it is clear that $J_{d_\mu}^{<a+1}=J_{d_\mu s_a}^{<a}$ because $a\notin J_{d_\mu}^{<a+1}$. Hence $N^\Lam(\mu s_{a},a+1)=N^\Lam(\mu s_{a},a)$.

Finally, suppose $k=a$. Then $\theta_a$ restricts to a bijection between $J_{d_\mu}^{<a}$ and $J_{d_\mu s_a}^{<a+1}\setminus\{a\}$. In this case it follows from definition that $N^\Lam(\mu,a)=N^\Lam(\mu s_{a},a+1)+\<\alpha_{\mu_{a+1}},h_{\mu_{a}}\>$.
\end{proof}

For each $1\leq t\leq p$, we set $\ell_t:=\<\Lam,\alpha_{\nu^t}\>$.

\begin{examp}\label{Example 2}
Let $\wnu=(1,1,2)$, $\mu=(2,1,1)$, then $d_{\mu}=s_2s_1$. By definition, we have $$N^\Lam(\mu,1)=\ell_{2},\,\,\,N^\Lam(\mu,2)=\ell_{1},\,\,\,N^\Lam(\mu,3)=\ell_{1}-1.$$
Now we consider $\mu\,s_{1}=(1,2,1)$. One can check directly that  $$
N^\Lam(\mu s_1,1)=\ell_{1},\,\,\,N^\Lam(\mu s_1,2)=\ell_{2}-\<\alpha_1,h_2\>,\,\,\, N^\Lam(\mu s_1,3)=\ell_{1}-1.  $$
\end{examp}

\begin{cor}\label{prop 3} Suppose that $N^\Lam(\mu,k)>0$ for any $1\leq k\leq n$. Let $1\leq t<n$ such that $d_\mu>d_\mu s_{t}$. Then the map $\phi_{t}:\,e(\wnu)\RR^\Lam({\beta})e(\mu)\rightarrow e(\wnu)\RR^\Lam({\beta})e(\mu s_{t})$ given by right multiplication of $\psi_{t}$ is injective. More generally, if $d_{\mu}=u_{1}u_{2}$ with $\ell(w)=\ell(u_{1})+\ell(u_{2})$, then the map $\phi_{u_{2}}:\,e(\wnu)\RR^\Lam({\beta})e(\mu)\rightarrow e(\wnu)\RR^\Lam({\beta})e(\mu u_{2}^{-1})$ given by right multiplication of $\psi_{u_{2}^{-1}}\,\,$ is injective.
\end{cor}

\begin{proof} By Lemma \ref{Lemma 12}, $d_{\mu s_{t}}=d_{\mu}s_{t}$. We can write $$
\psi_{d_{\mu}}e(\mu)=\psi_{d_{\mu}s_{t}}\psi_{s_{t}}e(\mu)=\psi_{d_{\mu s_{t}}}e(\mu s_{t})\psi_{s_{t}}. $$
The assumption that $N^\Lam(\mu,k)>0$ for any $1\leq k\leq n$ and Lemma \ref{Lemma 13} imply that $N^\Lam(\mu s_{t},k)>0$ for any $1\leq k\leq n$. Since $\psi_{t}\psi_{t}e(\mu s_{t})=Q_{\mu_{t+1},\mu_{t}}(x_{t},x_{t+1})e(\mu s_{t})$, it follows that for any $w\in\Sym(\mu,\wnu)$ and $r_k\in\N$, $1\leq k\leq n$, $\phi_t(\psi_{w}\prod_{k=1}^{n}x_{k}^{r_{k}}e(\mu))$ is of the form $\psi_{w s_{t}}\prod_{k=1}^{n}f_{k}e(\mu s_t)$, where
\[f_{k}=
\begin{cases}
x_{k}^{r_{k}}& k\neq t\,\,,t+1\\
x_{t}^{r_{t+1}} & k=t\\
x_{t+1}^{r_{t}}Q_{\nu_{t+1},\nu_{t}}(x_{t},x_{t+1})& k=t+1.\\
\end{cases}\]
Note that $f_{t+1}$ is a polynomial in $x_{t+1}$ of degree $r_t-\<\alpha_{\mu_{t+1}},h_{\mu_{t}}\>$ with leading coefficient in $K^\times$ and other coefficients in $K[x_{1},x_{2},\cdots,x_{t}]$. By Lemma \ref{keylem3}, we can write $\phi_t(\psi_{w}\prod_{k=1}^{n}x_{k}^{r_{k}}e(\mu))=c_0\psi_{w s_{t}}\prod_{k=1}^{n}x_{k}^{r'_{k}}e(\mu s_t)+\text{``lower terms''}$, where $c_0\in K^\times$ and ``lower terms'' means the degree of $x_{t+1}$ is less than $r_t-\<\alpha_{\mu_{t+1}},h_{\mu_{t}}\>$, and \[r'_{k}=
\begin{cases}
r_{k}& k\neq t\,\,,t+1\\
r_{t+1} & k=t\\
r_t-\<\alpha_{\mu_{t+1}},h_{\mu_{t}}\>& k=t+1.\\
\end{cases}\]
By Lemma \ref{Lemma 13}, if $k\neq t, t+1$, then $r'_{k}<N^\Lam(\mu s_{t},k)=N^\Lam(\mu,k)$ if $r_{k}<N^\Lam(\mu,k)$; and $r'_{t}=r_{t+1}<N^\Lam(\mu s_{t},t)=N^\Lam(\mu,t+1)$ if $r_{t+1}<N^\Lam(\mu,t+1)$; and $r'_{t+1}=r_{t}-\<\alpha_{\mu_{t+1}},h_{\mu_{t}}\><N^\Lam(\mu s_{t},t+1)$ if $r_{t}<N^\Lam(\mu,t)$.
By Theorem \ref{mainthm3b}, we know that $$\Bigl\{\psi_{w }\prod_{k=1}^{n}x_{k}^{r_{k}}e(\mu)\Bigm| w \in\Sym(\mu,\wnu), 0\leq r_{k}<N^\Lam(\mu,k), \forall\,1\leq k\leq n\Bigr\}$$
forms a $K$-basis of $e(\wnu)\RR^\Lam({\beta})e(\mu)$. Similarly, the set $$\Bigl\{\psi_{w s_{t}}\prod_{k=1}^{n}x_{k}^{r_{k}}e(\mu s_{t})\Bigm| w \in\Sym(\mu,\wnu), 0\leq r_{k}<N^\Lam(\mu s_{t},k), \forall\,1\leq k\leq n\Bigr\}$$
forms a $K$-basis of $e(\wnu)\RR^\Lam({\beta})e(\mu s_{t})$.

Now using Theorem \ref{mainthm3b} and Lemma \ref{keylem3}, we can see that the image of each basis element $\psi_{w s_{t}}\prod_{k=1}^{n}x_{k}^{r_{k}}e(\mu)$ under $\phi_t$ has a leading term and they are $K$-linearly independent. It follows that the image of those basis elements of $e(\wnu)\RR^\Lam({\beta})e(\mu)$ under $\phi_t$ are $K$-linearly independent, which implies that $\phi_t$ is injective.
%
\end{proof}

\subsection{The monomial bases of $\RR^\Lam(\beta)$ when $n=3$}

In this subsection, we shall completely determine a monomial basis for $\RR^\Lam(\beta)$ when $n=3$. Let $\beta\in Q_3^+$. Note that $\RR^\Lam(\beta)=\oplus_{\nu,\mu\in I^\beta}e(\mu)\RR^\Lam(\beta)e(\nu)$. By the results we have obtained in the last two subsections, we can assume without loss of generality that $\beta=2\alpha_1+\alpha_2$. We only need to construct a monomial basis for $e(1,2,1)\RR^\Lam(\beta)e(1,2,1)$. We set $\nu:=(1,2,1)$. Then
$\Sym(\nu,\nu)=\{(1), w:=(1,3)\}$, where $(1,3)$ denotes the transposition which swaps $1$ and $3$. We set $l_1:=\<\Lam,h_1\>, l_2:=\<\Lam,h_2\>$. Then we have $$\begin{aligned}
& N^\Lam(1,\nu,1)=l_{1}\,\,,N^\Lam(1,\nu,2)=l_{2}-a_{21}\,\,,N^\Lam(1,\nu,3)=l_{1}-a_{12}-2; \\
& N^\Lam(w,\nu,1)=l_{1}\,\,,N^\Lam(w,\nu,2)=l_{2}\,\,,N^\Lam(w,\nu,3)=l_{1}.
\end{aligned}
$$

\begin{lem}\label{iso2}
Suppose $\nu\,,\nu'\in I^{\beta}$, $1\leq t\leq n$ with $a_{\nu_{t},\nu_{t+1}}=0$. Then the map $\phi:\,e(\nu')\RR^\Lam({\beta})e(\nu)\rightarrow e(\nu')\RR^\Lam({\beta})e(\nu s_{t})$ given by right multiplication of $\psi_{t}$ is an isomorphism.
\end{lem}

\begin{proof} This is clear because $\psi_t^2e(\nu)=e(\nu)$ by assumption.
\end{proof}

Suppose $a_{12}=0$, then $a_{21}=0$. Applying Corollary \ref{maincor1} we can get that $$
\dim\,e(1,2,1)\RR^\Lam({\beta})e(1,2,1)=2l_{1}(l_{1}-1)l_{2}, $$
which is exactly the same as the dimension of $e(1,2,1)\RR^\Lam({\beta})e(1,1,2)$. Now using Lemma \ref{iso2}, one can easily get a monomial basis of $e(1,2,1)\RR^\Lam({\beta})e(1,2,1)$ from the known monomial basis (see Theorem \ref{mainthm3b}) of $e(1,2,1)\RR^\Lam({\beta})e(1,1,2)$ in this case.

Henceforth we assume $a_{12}\neq 0$ and thus $a_{12}\leq -1\geq a_{21}$. By definition, we have $a^{\Lambda}_{1}(x_{1})e(1,2,1)=0$, which implies that  \begin{equation}\label{11a}
x_1^{l_1}e(1,2,1)\in\text{$K$-Span}\{x_1^{c_1}e(1,2,1)|0\leq c_1<l_1\} .
\end{equation}
Similarly, \begin{equation}\label{11b}
Q_{1,2}(x_{1},x_{2})a^{\Lambda}_{2}(x_{2})e(1,2,1)=\psi_{1}a^{\Lambda}_{2}(x_{1})\psi_{1}e(1,2,1)=0,
\end{equation}
which implies that \begin{equation}\label{11c}
x_2^{l_2-a_{21}}e(1,2,1)\in\text{$K$-Span}\{x_1^{c_1}x_2^{c_2}e(1,2,1)|c_1\geq 0, 0\leq c_2<l_2-a_{21}\} .
\end{equation}

Similarly, $\psi_{1}\psi_{2}\psi_1a^{\Lambda}_{2}(x_{2})e(1,2,1)=\psi_1\psi_2a_2^\Lam(x_1)\psi_1e(1,2,1)=0$ together with $$\psi_{1}\psi_{2}\psi_1a^{\Lambda}_{1}(x_{1})e(1,2,1)=0, $$ imply that \begin{equation}\label{12a}
\psi_1\psi_2\psi_1x_1^{l_1}x_2^{l_2}e(1,2,1)\in\text{$K$-Span}\{\psi_1\psi_2\psi_1x_1^{c_1}x_2^{c_2}e(1,2,1)|0\leq c_1<l_1,\,0\leq c_2<l_2\} .
\end{equation}
As a result, we have that for any $a_1,a_2\in\N$, $$\begin{aligned}
& x_1^{a_1}x_2^{a_2}e(1,2,1)\in\text{$K$-Span}\{x_1^{c_1}x_2^{c_2}e(1,2,1)|0\leq c_1<l_1, 0\leq c_2<l_2-a_{21}\},\\
&\psi_{1}\psi_{2}\psi_{1}x_1^{a_1}x_2^{a_2}e(1,2,1)\in\text{$K$-Span}\{\psi_{1}\psi_{2}\psi_{1}x_1^{c_1}x_2^{c_2}e(1,2,1)|0\leq c_1<l_1, 0\leq c_2<l_2\} .
\end{aligned}
$$

Following \cite[(3.4)]{KK}, we define $$
\overline{Q}_{1,2,3}=\sum_{\mu\in I^3,\mu_1=\mu_3}\frac{Q_{\mu_1,\mu_2}(x_1,x_2)-Q_{\mu_1,\mu_2}(x_3,x_2)}{x_1-x_3}e(\mu) .
$$
Applying \cite[(3.7)]{KK}, we can deduce that \begin{equation}\label{15a}
\psi_{1}\psi_{2}\psi_{1}a^{\Lambda}_{1}(x_{3})e(1,2,1)-Q_{1,2}(x_{1},x_{2})s_{1}(\partial_{2}a_{1}(x_{2}))e(1,2,1)
=a^{\Lambda}_{1}(x_{1})\psi_{1}\psi_{2}\psi_{1}e(1,2,1)=0,
\end{equation}
Note that the degree of $x_3$ in $a^{\Lambda}_{1}(x_{3})$ is $l_1$, while the degree of $x_3$ in $Q_{1,2}(x_{1},x_{2})s_{1}(\partial_{2}a_{1}(x_{2}))$ is $l_1-1$. Moreover, the coefficient of $x_3^{l_1}$ in $a^{\Lambda}_{1}(x_{3})$ is in $K^\times$. Similarly, applying \cite[(3.7)]{KK} and the above definition, we can get that \begin{equation}\label{15b}
\psi_{1}\psi_{2}\psi_{1}s_{1}(\partial_{2}a_{1}(x_{2}))e(1,2,1)+\overline{Q}_{1,2,3}s_{1}(\partial_{2}a_{1}(x_{2}))e(1,2,1)
=\psi_{2}\psi_{1}a^{\Lambda}_{1}(x_{1})\psi_{1}\psi_{2}e(1,2,1)=0.
\end{equation}
Note the degree of $x_3$ in $s_{1}(\partial_{2}a_{1}(x_{2}))$ is $l_1-1$, while the degree of $x_3$ in $\overline{Q}_{1,2,3}s_{1}(\partial_{2}a_{1}(x_{2}))$ is $l_1-a_{12}-2\geq l_1-1$. Moreover, the coefficient of $x_3^{l_1-1}$ in $s_{1}(\partial_{2}a_{1}(x_{2}))$ is in $K^\times$, and the coefficient of $x_3^{l_1-a_{12}-2}$ in $\overline{Q}_{1,2,3}s_{1}(\partial_{2}a_{1}(x_{2}))$ is in $K^\times$ too.

Using (\ref{15a}), (\ref{15b}) and the two displayed equalities in the last paragraph, we can deduce that the following result.

\begin{thm}\label{432} Suppose that $a_{1,2}\neq 0$ and $\beta=2\alpha_1+\alpha_2$. Then the following subset $$\begin{aligned}
&\{\psi_{1}\psi_{2}\psi_{1}x_{1}^{k_{1}}x_{2}^{k_{2}}x_{3}^{k_{3}}|k_{1}<l_{1},\,\,k_{2}<l_{2},\,\,k_{3}<l_{1}\}\\
&\qquad \bigcup
\{x_{1}^{k_{1}}x_{2}^{k_{2}}x_{3}^{k_{3}}|k_{1}<l_{1},\,\,k_{2}<l_{2}-a_{21},\,\,k_{3}<l_{1}-a_{12}-2\},\end{aligned}
$$
forms a $K$-basis of $e(1,2,1)\RR^\Lam(\beta)e(1,2,1)$, where $l_1=\<\Lam,h_1\>$, $l_2=\<\Lam,h_2\>$.
\end{thm}

\begin{proof} By the discussion before the theorem, we see that the elements in the above subset are $K$-linear generators of $e(1,2,1)\RR^\Lam(\beta)e(1,2,1)$. Using dimension formula Corollary \ref{maincor1}, we see this subset has the same cardinality as the dimension of $e(1,2,1)\RR^\Lam(\beta)e(1,2,1)$. Thus it must form a $K$-basis of $e(1,2,1)\RR^\Lam(\beta)e(1,2,1)$.
This completes the proof of the theorem.
\end{proof}

\begin{rem}
When $a_{12}=0$, the set in Theorem \ref{432} will not be a $K$-linear basis of $e(1,2,1)\RR^\Lam(\beta)e(1,2,1)$. Actually, Lemma \ref{keylem3} tells us the following set is $K$-linearly dependent in $e(2,1,1)\RR^\Lam(\beta)e(1,2,1)$ :$$\{\psi_{2}\psi_{1}x_{1}^{k_{1}}x_{2}^{k_{2}}x_{3}^{k_{3}}|k_{1}<l_{1},\,\,k_{2}<l_{2},\,\,k_{3}<l_{1}\}.
$$ Hence, $$\{\psi_{1}\psi_{2}\psi_{1}x_{1}^{k_{1}}x_{2}^{k_{2}}x_{3}^{k_{3}}|k_{1}<l_{1},\,\,k_{2}<l_{2},\,\,k_{3}<l_{1}\}$$
is $K$-linearly dependent too.
\end{rem}

\subsection{Some counter-examples on the graded freeness of $\RR^\Lam(n)$ over its subalgebra $\RR^\Lam(m)$ with $m<n$}

Let $\beta\in Q_n^+$ and $i\in I$ such that $e(\beta,i)\neq 0$. Kang and Kashiwara (\cite[Theorem 4.5]{KK}) have shown that $\RR^\Lam(\beta+\alpha_i)e(\beta,i)$ is a projective right $\RR^\Lam(\beta)$-module. It follows that
(\cite[Remark 4.20(ii)]{KK}) $\RR^\Lam(n)$ is a projective $\RR^\Lam(m)$-module when $n\geq m$, where $$\RR^\Lam(n)=\oplus_{\beta\in Q_n^+}\RR^\Lam(\beta).
$$
It is natural to ask whether $\RR^\Lam(n)$ is a free $\RR^\Lam(m)$-module. Moreover, when it is a free module, one can ask whether $\RR^\Lam(n)$ has a homogeneous basis over the subalgebra $\RR^\Lam(m)$.
In this subsection, we shall use our main results Theorem \ref{mainthmA} and Corollary \ref{maincor1} to give some examples to show that the answers to these questions are negative in general.

\begin{examp}
Let $A$ be of type $A^{(1)}_1$, i.e. $$A=\begin{pmatrix}2\,&-2\\
-2\,&2\end{pmatrix}.
$$
Assume $\Lam=\Lam_1+2\Lam_2$. By the Brundan-Kleshchev's isomorphism \cite{BK:GradedKL} and the Ariki-Koike bases for the cyclotomic Hecke algebras \cite{AK}, it is easy to see that $\RR^\Lam(2)$ is a free right $\RR^\Lam(1)$-module. However, using Theorem \ref{mainthmA}, we can get that $$\begin{aligned}
\dim_q\,\RR^\Lam(1)&=\dim_q\,\RR^\Lam(\alpha_1)+\dim_q\,\RR^\Lam(\alpha_2)\\
&=1+(1+q^2)=2+q^2 ,
\end{aligned}
$$
while $$\begin{aligned}
&\quad\,\dim_q\,\RR^\Lam(2)\\
&=\dim_q\,\RR^\Lam(2\alpha_1)+\dim_q\,\RR^\Lam(2\alpha_2)+\dim_q\,e(1,2)\RR^\Lam(\alpha_1+\alpha_2)e(1,2)\\
&\qquad +\dim_q\,e(1,2)\RR^\Lam(\alpha_1+\alpha_2)e(2,1)+\dim_q\,e(2,1)\RR^\Lam(\alpha_1+\alpha_2)e(1,2)\\
&\qquad\qquad +\dim_q\,e(2,1)\RR^\Lam(\alpha_1+\alpha_2)e(2,1)\\
&=0+(q^{-2}+2+q^2)+(1+q^2+q^4+q^6)+2(q^2+q^4)+(1+2q^2+2q^4+q^6)\\
&=2q^6+5q^4+6q^2+4+q^{-2}.
\end{aligned}
$$
This implies that $\dim_q\,\RR^\Lam(1)$ is not a factor of $\dim_q\,\RR^\Lam(2)$. Thus, as a free right $\RR^\Lam(1)$-module, $\RR^\Lam(2)$ does not have a homogeneous basis.
\end{examp}

\begin{examp}
Let $A$ be of type $A_2$, i.e. $$A=\begin{pmatrix}2\,&-1\\
-1\,&2\end{pmatrix}.
$$
Assume $\Lam=\Lam_1+\Lam_2$, $\beta=\alpha_1+\alpha_2$. Using Corollary \ref{maincor1}, we can get that $$\begin{aligned}
\dim\,\RR^\Lam(\beta)&=\dim\,e(1\,2)\RR^\Lam(\beta)e(1\,2)+\dim\,e(1\,2)\RR^\Lam(\beta)e(2\,1)\\
&\qquad +\dim\,e(2\,1)\RR^\Lam(\beta)e(1\,2)+\dim\,e(2\,1)\RR^\Lam(\beta)e(2\,1)\\
&=2+1+1+2=6.
\end{aligned}
$$
Similarly, $$\begin{aligned}
&\quad\,\dim\,\RR^\Lam(\beta+\alpha_1)e(\beta,1)\\
&=\dim\,\RR^\Lam(\beta+\alpha_1)e(1,2,1)+\RR^\Lam(\beta+\alpha_1)e(2,1,1)\\
&=\dim\,e(2,1,1)\RR^\Lam(\beta+\alpha_1)e(1,2,1)+\dim\,e(1,2,1)\RR^\Lam(\beta+\alpha_1)e(1,2,1)\\
&\qquad +\dim\,e(1,1,2)\RR^\Lam(\beta+\alpha_1)e(1,2,1)+\dim\,e(2,1,1)\RR^\Lam(\beta+\alpha_1)e(2,1,1)\\
&\qquad\qquad +\dim\,e(1,2,1)\RR^\Lam(\beta+\alpha_1)e(2,1,1)+\dim\,e(1,1,2)\RR^\Lam(\beta+\alpha_1)e(2,1,1)\\
&=2+1+0+4+2+0=9.
\end{aligned}
$$
Since $6\nmid 9$, it follows that $\RR^\Lam(\beta+\alpha_1)e(\beta,1)$ is not a free right $\RR^\Lam(\beta)$-module.
\end{examp}

\begin{examp}
Let $A$ be of type $A_3$, i.e. $$A=\begin{pmatrix}2\,&-1\,&0\\
-1\,&2\,&-1\\
0\,&-1\,&2\end{pmatrix}.
$$
Assume $\Lam=3\Lam_1+2\Lam_2+2\Lam_3$. Using Corollary \ref{maincor1}, we can get that $$\dim\,\RR^\Lam(1)=3+2+2=7,
$$
and $$\begin{aligned}
\dim\,\RR^\Lam(2)&=\dim\,\RR^\Lam(2\alpha_1)+\dim\,\RR^\Lam(2\alpha_2)+\dim\,\RR^\Lam(2\alpha_3)\\
&\qquad +\dim\,\RR^\Lam(\alpha_1+\alpha_2)+\dim\,\RR^\Lam(\alpha_1+\alpha_3)+\dim\,\RR^\Lam(\alpha_2+\alpha_3)\\
&=12+4+4+29+24+20=93.
\end{aligned}
$$
Again, we conclude that $\RR^\Lam(2)$ is not a free $\RR^\Lam(1)$-module.
\end{examp}


Let $\beta\in Q_n^+$. For each $i\in I$, there is a natural map $\gamma_{\beta,i}: \RR^\Lam(\beta)\rightarrow e(\beta,i)\RR^\Lam(\beta+\alpha_i)e(\beta,i)$. We define $$
\gamma_\beta=\oplus_{i\in I}\gamma_{\beta,i}:\,\RR^\Lam(\beta)\,\rightarrow\,\oplus_{i\in I}e(\beta,i)\RR^\Lam(\beta+\alpha_i)e(\beta,i),
$$
This map was studied in \cite{ZH} and was proved to be injective except in some special cases. It is natural to expect that $\oplus_{i\in I}e(\beta,i)\RR^\Lam(\beta+\alpha_i)e(\beta,i)$ is a free $\RR^\Lam(\beta)$-module when $\gamma_\beta$ is injective. The following example shows that this again fails in general.

\begin{examp}
Let $A$ be of type $A_2$, $\beta=\alpha_1+\alpha_2$ and $\Lam=3\Lam_1+2\Lam_2$. Then $$
\Lam-w_0\Lam=5(\alpha_1+\alpha_2)\neq\beta .
$$
It follows from \cite[Theorem 3.7]{ZH} that $\gamma_\beta$ is injective in this case. However, using Corollary \ref{maincor1}, we can get that $$\begin{aligned}
\dim\,\RR^\Lam(\beta)&=\dim\,e(1,2)\RR^\Lam(\beta)e(1,2)+\dim\,e(1,2)\RR^\Lam(\beta)e(2,1)\\
&\qquad +\dim\,e(2,1)\RR^\Lam(\beta)e(1,2)+\dim\,e(2,1)\RR^\Lam(\beta)e(2,1)\\
&=9+6+6+8=29,
\end{aligned}
$$
and $$
\begin{aligned}
&\quad\,\dim\,e(\beta,1)\RR^\Lam(\beta+\alpha_i) e(\beta,1)+\dim\,e(\beta,2)\RR^\Lam(\beta+\alpha_i) e(\beta,2)\\
&=\dim\,e(1,2,1)\RR^\Lam(\beta+\alpha_i) e(1,2,1)+\dim\,e(1,2,1)\RR^\Lam(\beta+\alpha_i) e(2,1,1)\\
&\qquad +\dim\,e(2,1,1)\RR^\Lam(\beta+\alpha_i) e(1,2,1)+\dim\,e(2,1,1)\RR^\Lam(\beta+\alpha_i) e(2,1,1)\\
&\qquad\quad +\dim\,e(1,2,2)\RR^\Lam(\beta+\alpha_i) e(1,2,2)+\dim\,e(1,2,2)\RR^\Lam(\beta+\alpha_i) e(2,1,2)\\
&\qquad\qquad +\dim\,e(2,1,2)\RR^\Lam(\beta+\alpha_i) e(1,2,2)+\dim\,e(2,1,2)\RR^\Lam(\beta+\alpha_i) e(2,1,2)\\
&=36+36+36+48+36+24+24+20=260.
\end{aligned}
$$
Note that $29\nmid 260$. It follows that  $\oplus_{i\in I}e(\beta,i)\RR^\Lam(\beta+\alpha_i)e(\beta,i)$ is not a free right $\RR^\Lam(\beta)$-module.
\end{examp}

The above examples imply that in general one can not construct a basis of the cyclotomic quiver Hecke algebra $\RR^\Lam(\beta)$ inductively via the injection $\gamma_\beta$.

\end{document}